\documentclass[a4paper,UKenglish,cleveref,autoref,thm-restate]{lipics-v2021}

\nolinenumbers

\usepackage{anyfontsize}

\usepackage{microtype}
\usepackage{nccmath}
\usepackage[nocompress]{cite}
\usepackage{tikz}
\usetikzlibrary{automata, positioning, patterns.meta}
\usepackage{tabularx}
\usepackage{booktabs}
\usepackage{xspace}
\usepackage[noEnd=true,indLines=false]{algpseudocodex}

\usepackage[only,llbracket,rrbracket]{stmaryrd}
\usepackage{mathrsfs}
\usepackage{mathtools}
\usepackage{adjustbox}
\usepackage{todonotes}

\renewcommand{\geq}{\geqslant}
\renewcommand{\leq}{\leqslant}

\newcommand{\ie}{i.e.,~}

\newcommand{\sse}{\subseteq}

\DeclareMathOperator{\polylog}{polylog}

\newcommand{\abs}[1] {\ensuremath\left|#1\right|}

\newcommand{\N}{\mathbb{N}}

\newcommand{\bigO}{\mathcal{O}}

\newcommand{\gen}[2][]{\langle #2 \rangle}
\newcommand{\ncl}[2][]{\langle\!\langle #2 \rangle\!\rangle}

\makeatletter
\newcommand{\@CCLASS}[1]{\hbox{\small$\mathsf{#1}$}}

\newcommand{\NPOLYLOGTIME}{\ensuremath{\@CCLASS{NPOLYLOGTIME}}\xspace}
\newcommand{\DPOLYLOGTIME}{\ensuremath{\@CCLASS{DPOLYLOGTIME}}\xspace}
\newcommand{\SigkPOLYLOGTIME}[1]{\ensuremath{\@CCLASS{\Sigma}_{#1}\@CCLASS{POLYLOGTIME}}\xspace}
\newcommand{\DLOGTIME}{\ensuremath{\@CCLASS{DLOGTIME}}\xspace}
\newcommand{\FO}{\ensuremath{\@CCLASS{FO}}\xspace}
\newcommand{\FOLL}{\ensuremath{\@CCLASS{FOLL}}\xspace}
\newcommand{\AC}{\ensuremath{\@CCLASS{AC}\xspace}}
\newcommand{\ACz}{\ensuremath{\@CCLASS{AC}^0}\xspace}
\newcommand{\qACz}{\ensuremath{\@CCLASS{qAC}^0}\xspace}
\newcommand{\NC}{\ensuremath{\@CCLASS{NC}}\xspace}
\newcommand{\ACC}{\ensuremath{\@CCLASS{ACC}}\xspace}
\newcommand{\qAC}{\ensuremath{\@CCLASS{qAC}}\xspace}
\newcommand{\TC}{\ensuremath{\@CCLASS{TC}}\xspace}
\newcommand{\LOGSPACE}{\ensuremath{\@CCLASS{L}}\xspace}
\newcommand{\NL}{\ensuremath{\@CCLASS{NL}}\xspace}
\newcommand{\SL}{\ensuremath{\@CCLASS{SL}}\xspace}
\newcommand{\PTIME}{\ensuremath{\@CCLASS{P}}\xspace}
\newcommand{\SigkPTIME}[1]{\ensuremath{\@CCLASS{\Sigma}_{#1}\@CCLASS{P}}\xspace}
\newcommand{\NP}{\ensuremath{\@CCLASS{NP}}\xspace}
\newcommand{\GI}{\ensuremath{\@CCLASS{GI}}}
\newcommand{\PSPACE}{\ensuremath{\@CCLASS{PSPACE}}\xspace}
\newcommand{\LH}{\ensuremath{\@CCLASS{LH}}\xspace}
\newcommand{\PLH}{\ensuremath{\@CCLASS{PolyLH}}\xspace}
\newcommand{\DPLL}{\ensuremath{\@CCLASS{DPLL}}\xspace}
\renewcommand{\L}{\LOGSPACE}
\newcommand{\DTISP}{\ensuremath{\@CCLASS{DTISP}}\xspace}
\newcommand{\NTISP}{\ensuremath{\@CCLASS{NTISP}}\xspace}
\newcommand{\DTIME}{\ensuremath{\@CCLASS{DTIME}}\xspace}
\newcommand{\NTIME}{\ensuremath{\@CCLASS{NTIME}}\xspace}

\makeatother

\usepackage{relsize}

\newcommand{\pv}[1]{\mathbf{#1}}
\newcommand{\pvI}[1]{\llbracket#1\rrbracket}
\newcommand{\pvL}{\mathbb{L}}

\newcommand{\pvM}{\mathbin{\raisebox{.4pt}{\textup{\textcircled{\raisebox{.6pt}{\smaller[2]\textit{m}}}}}}}

\newcommand{\GAb}{\pv{G}_\textup{\textbf{Ab}}}
\newcommand{\Gsol}{\pv{G}_\textup{\textbf{sol}}}

\newcommand*{\gR}[1][]{\mathrel{\mathcal{R}_{#1}}}

\newcommand\nindent{.5pt}
\newcommand\noverline[1]{%
	\kern\nindent\overline{\kern-\nindent#1\kern-\nindent}\kern\nindent}

\newcommand{\normalop}[3]{\operatorname{NC}_{#3}(#1; #2)}
\newcommand{\commop}[3][]{\operatorname{PCC}_{#1}(#2;#3)}

\theoremstyle{plain}
\newtheorem*{rlemma}{Reachability Lemma}
\newtheorem{question}[theorem]{Question}

\makeatletter 
\def\th@remark{%
  \thm@headfont{%
    \textcolor{lipicsGray}{$\blacktriangleright$}\nobreakspace\sffamily\bfseries}%
  \normalfont 
  \thm@preskip\topsep \divide\thm@preskip\tw@
  \thm@postskip\thm@preskip
}
\makeatother

\theoremstyle{plain}

\makeatletter
\providecommand\iitem{}
\providecommand\qitem{}
\newcommand\decproblem@iitem@label{\rlap{Input.}\phantom{Question.}}
\newcommand\decproblem@qitem@label{Question.}
\newenvironment{decproblem}{%
  \begin{description}\begin{samepage}%
  \renewcommand{\iitem}{\item[\decproblem@iitem@label]}%
  \renewcommand{\qitem}{\item[\decproblem@qitem@label]}%
}{%
  \end{samepage}\end{description}%
}
\makeatother

\newcommand{\dMemb}[2][]{\textup{\textsc{memb${}_{\mathbf{#1}}\expandafter\ifx\expandafter\relax\detokenize{#2}\relax\else(#2)\fi$}}}
\newcommand{\dConj}[2][]{\textup{\textsc{conj${}_{\mathbf{#1}}\expandafter\ifx\expandafter\relax\detokenize{#2}\relax\else(#2)\fi$}}}
\newcommand{\dMembS}[2][]{\textup{\textsc{memb${}^{\sharp}_{\mathbf{#1}}\expandafter\ifx\expandafter\relax\detokenize{#2}\relax\else(#2)\fi$}}}
\newcommand{\dConjS}[2][]{\textup{\textsc{conj${}^{\sharp}_{\mathbf{#1}}\expandafter\ifx\expandafter\relax\detokenize{#2}\relax\else(#2)\fi$}}}
\newcommand{\dMGS}[2][]{\textup{\textsc{mgs${}_{\mathbf{#1}}\expandafter\ifx\expandafter\relax\detokenize{#2}\relax\else(#2)\fi$}}}
\newcommand{\dRequiv}[2][]{\textup{\textsc{\ensuremath{\gR}-equiv${}_{\mathbf{#1}}\expandafter\ifx\expandafter\relax\detokenize{#2}\relax\else(#2)\fi$}}}
\newcommand{\dEqn}[2][]{\textup{\textsc{eqn${}_{\mathbf{#1}}\expandafter\ifx\expandafter\relax\detokenize{#2}\relax\else(#2)\fi$}}}
\newcommand{\dEqnSys}[2][]{\textup{\textsc{eqn${}^\ast_{\mathbf{#1}}\expandafter\ifx\expandafter\relax\detokenize{#2}\relax\else(#2)\fi$}}}

\title{Efficient Compression in Semigroups}
\titlerunning{Efficient Compression in Semigroups}

\author{Florian Stober}{FMI, University of Stuttgart \\ Universitätsstraße 38, 70569 Stuttgart, Germany}{florian.stober@fmi.uni-stuttgart.de}{https://orcid.org/0000-0002-5516-6660}{}
\author{Alexander Thumm}{University of Siegen \\ Hölderlinstraße 3, 57076 Siegen, Germany}{alexander.thumm@uni-siegen.de}{https://orcid.org/0009-0005-4240-2045}{}
\author{Armin Weiß}{FMI, University of Stuttgart \\ Universitätsstraße 38, 70569 Stuttgart, Germany}{armin.weiss@fmi.uni-stuttgart.de}{https://orcid.org/0000-0002-7645-5867}{}
\authorrunning{F.~Stober, A.~Thumm, A.~Weiß}

\funding{This work was partially supported by the Deutsche Forschungsgemeinschaft (DFG, German Research Foundation) -- LO~748/15-1, WE~6835/1-2.}
\acknowledgements{The authors thank Markus Lohrey for valuable discussions, and the anonymous referees of STACS~2026 for their helpful comments and suggestions.}
\hideLIPIcs

\Copyright{Alexander Thumm, Armin Weiß}


\begin{document}

\maketitle

\begin{abstract}
  Straight-line programs are a central tool in several areas of computer science, including data compression, algebraic complexity theory, and the algorithmic solution of algebraic equations.
  In the algebraic setting, where straight-line programs can be interpreted as circuits over algebraic structures such as semigroups or groups, they have led to deep insights in computational complexity.

  A key result by Babai and Szemerédi (1984) showed that finite groups afford efficient compression via straight-line programs, enabling the design of a black-box computation model for groups.
  Building on their result, Fleischer (2019) placed the Cayley table membership problem for certain classes (pseudovarieties) of finite semigroups in \NPOLYLOGTIME, and in some cases even in \FOLL. 
  He also provided a complete classification of pseudovarieties of finite \emph{monoids} affording efficient compression.

  In this work, we complete this classification program initiated by Fleischer, characterizing precisely those pseudovarieties of finite \emph{semigroups} that afford efficient compression via straight-line programs. 
  Along the way, we also improve several known bounds on the length and width of straight-line programs over semigroups, monoids, and groups.
  These results lead to new upper bounds for the membership problem in the Cayley table model: for all pseudovarieties that afford efficient compression and do not contain any nonsolvable group, we obtain \FOLL algorithms.
  In particular, we resolve a conjecture of Barrington, Kadau, Lange, and McKenzie (2001), showing that the membership problem for all solvable groups is in \FOLL.
\end{abstract}

\section{Introduction}

The \emph{membership problem} asks, given a (finite) algebraic structure $S$, a set $\Sigma \sse S$, and a target element $t \in S$, whether $t$ belongs to the substructure of $S$ generated by $\Sigma$. 
Thinking of groups, semigroups, or vector spaces, this is a very fundamental problem in computational algebra with many applications.
For permutation groups, Sims gave an efficient solution already in 1967 \cite{Sims67}, later refined to an \NC algorithm by Babai, Luks, and Seress \cite{BabaiLS87}.
For transformation semigroups, Kozen \cite{koz77} showed in 1977 that the problem is \PSPACE-complete, as hard as intersection non-emptiness for deterministic finite automata (DFAs). 

As a different variant, the membership problem \dMemb[CT]{} in the \emph{Cayley table model} was introduced by Jones, Lien, and Laaser \cite{JonesLL76}, where the semigroup is given by its multiplication table; here the problem is \NL-complete.
For groups, Barrington and McKenzie \cite{BarringtonM91} showed that $\dMemb[CT]{\pv{G}}$ can be solved in \L with an oracle to undirected graph reachability \cite{Reingold08}, and conjectured it might be $\LOGSPACE$-hard.
Fleischer \cite{Fleischer19diss,Fleischer22} refuted the latter (under \ACz-reductions) by placing the problem in \NPOLYLOGTIME. 
His proof is based on \emph{straight-line programs} (algebraic circuits or context-free grammars producing precisely one word), a tool central in data compression~\cite{KiefferY00,Bannai16}, algebraic complexity~\cite{BuergisserCS97}, and in solving algebraic equations~\cite{Jez15,CiobanuDE16}. 
A key feature is their support for efficient manipulation of compressed data \cite{Lohrey2012survey,Lohrey14compressed,GanardiJL19,VanderhoevenL25}.

Babai and Szemer\'edi \cite{BabaiS84} showed that finite groups afford \emph{efficient compression}: every element can be expressed by a straight-line program of polylogarithmic length in the size of the group.
Fleischer used this to place $\dMemb[CT]{\pv{G}}$ in \NPOLYLOGTIME and extended the result to pseudovarieties of monoids: efficient compression occurs precisely for Clifford monoids (which comprise both groups and semilattices) and commutative monoids.
In contrast, there is no maximal pseudovariety of semigroups that affords efficient compression.\footnote{At the first glance, the difference between monoids and semigroups might seem negligible; however, the landscape of pseudovarieties of semigroups is much richer than the one of monoids.}

The membership problem has also been studied restricted to other pseudovarieties: Beaudry, McKenzie, and Thérien \cite{BeaudryMT92} investigated aperiodic monoids, while Fleischer and the authors~\cite{FleischerSTW25} considered inverse semigroups.

An important variant are straight-line programs of polylogarithmic length and \emph{bounded width}. 
Fleischer showed that they yield membership algorithms in \FOLL (polynomial-size Boolean circuits of depth $\log \log n$), which applies to all commutative semigroups.
Earlier, Barrington, Kadau, Lange, and McKenzie \cite{BarringtonKLM01} placed membership in solvable groups of bounded derived length in \FOLL, and conjectured that the membership problem for \emph{all} solvable groups may also be in \FOLL.
This was partially confirmed by Collins, Grochow, Levet, and the third author~\cite{CollinsGLW25} showing this to be true for the class of all nilpotent groups.

In this work, we complete Fleischer's program to characterize pseudovarieties of finite semigroups that afford efficient compression.
Moreover, we also improve upon some of the best previously-known length and width bounds for semigroups, monoids, and groups.
Finally, we apply our findings to the membership problem, in particular, resolving Barrington, Kadau, Lange, and McKenzie's conjecture.
In more detail, our results are as follows.

\vspace{-2mm}

\subparagraph*{Our Contribution.}

Our main theorem completely characterizes those pseudovarieties $\pv{V}$ of semigroups that afford efficient compression~--~meaning that, for some $k \in \N$, all $S \in \pv{V}$ of size~$\abs{S} \leq N$ admit straight-line programs of length $\mathcal{O}(\smash{\log}^k N)$; see \cref{sec:compression}.
Here, the following three pseudovarieties~--~each requiring straight-line programs of length $\Omega(\sqrt{N})$~--~play a crucial role, since they serve as primary obstructions:
\begin{equation*}
  \pv{LRB} = \pvI{x^2 \approx x, xyx \approx xy},
  \quad 
  \pv{RRB} = \pvI{x^2 \approx x, xyx \approx yx}, 
  \quad
  \pv{T} = \pvI{x^2 \approx xyx \approx 0}.
\end{equation*}

\begin{theorem}\label{thm:main-intro}
	Let $\pv{V}$ be a pseudovariety of semigroups.
	The following are equivalent.
	\begin{bracketenumerate}
		\item The pseudovariety $\pv{V}$ affords efficient compression.
		\item The pseudovariety $\pv{V}$ contains neither $\pv{LRB}$, $\pv{RRB}$, nor $\pv{T}$.
		\item The pseudovariety $\pv{V}$ admits straight-line programs of length $\mathcal{O}(\smash{\log}^2 N)$.
	\end{bracketenumerate}
  Furthermore, if all groups in $\pv{V}$ are solvable, then the above are equivalent to $\pv{V}$ admitting straight-line programs of length $\mathcal{O}(\log N)$ as well as of width $\mathcal{O}(1)$ and length $\mathcal{O}(\polylog N)$.
\end{theorem}

Moreover, if $\pv{V} \subseteq \pv{RB} \pvM \pv{N}_k$ for some $k \geq 1$, then the pseudovariety admits straight-line programs of bounded width and length.
Except in the case that $\pv{V}$ contains a nonsolvable group, we show that these bounds are essentially asymptotically optimal.

In addition, we also prove that the lower bounds for pseudovarieties containing at least one of the obstructions $\pv{LRB}$, $\pv{RRB}$, and $\pv{T}$ are also asymptotically optimal.

\begin{proposition}\label{thm:main-squareroot}
  All finite semigroups admit straight-line programs of length $\mathcal{O}(\sqrt{N})$.
\end{proposition}

Our proofs are fully constructive and can be found in Sections~\ref{sec:obstructions}--\ref{sec:general}.

\medskip
In \cref{sec:membership}, we apply our findings to the membership problem proving the following result, where \dMemb[CT]{\pv{V}} denotes the membership problem for $\pv{V}$ in the Cayley table model.

\begin{corollary}\label{cor:main_membership}
Let $\pv{V}$ be a pseudovariety of semigroups with $\pv{LRB}$, $\pv{RRB}$, $\pv{T} \not \sse \pv{V}$.
\begin{bracketenumerate}
\item The membership problem \dMemb[CT]{\pv{V}} is in $\NPOLYLOGTIME \sse \qACz$.
	\item If, moreover, $\pv{V}$ contains no nonsolvable group, then \dMemb[CT]{\pv{V}} is in \FOLL. 
\end{bracketenumerate}
\end{corollary}

Our results almost completely answer an open problem due to Fleischer \cite{Fleischer22}, who deemed it ``interesting to see whether the Cayley semigroup membership problem
can be shown to be in \FOLL for all classes of semigroups with the polylogarithmic circuits
property.''
Moreover, we positively resolve Barrington, Kadau, Lange, and McKenzie's conjecture \cite{BarringtonKLM01} concerning the membership problem for the pseudovariety $\Gsol$ of all finite solvable groups.

\begin{corollary}
 The problem \dMemb[CT]{\Gsol} is in \FOLL. 
\end{corollary}

\subparagraph*{Outline of the Proof.}

Our proof of \cref{thm:main-intro} combines structural results on semigroups with explicit constructions.
The negative results for $\pv{LRB}$, $\pv{RRB}$, and $\pv{T}$ were already established by Fleischer~\cite{Fleischer19diss}.
For the positive direction, we rely on a recent characterization of pseudovarieties $\pv{V}$ satisfying $\pv{T} \not\sse \pv{V}$ due to the second author~\cite{Thumm2025}, which shows that any such pseudovariety necessarily falls into one (or both) of the following two classes.

\begin{itemize}
  \item The pseudovariety $\pv{V}$ is \emph{almost completely regular}, meaning that all its members satisfy an identity of the form $x_1 \cdots x_n \approx x_1 \cdots x_{i-1} (x_i \cdots x_j)^{\omega+1} x_{j+1} \cdots x_n$.
    This condition properly generalizes complete regularity~--~equivalently, being a union of groups~--~which is characterized by the identity $x \approx x^{\omega + 1}$.
    In \cref{sec:general} we show that the general problem for almost completely regular pseudovarieties reduces to this special case.
    If, in addition, the pseudovariety satisfies $\pv{LRB}, \pv{RRB} \not\sse \pv{V}$, then the completely regular members of $\pv{V}$ are necessarily normal bands of groups \cite[Proposition~4]{Rasin81}.
    Exploiting this structural restriction, we further reduce the problem to the group case in \cref{sec:completely_regular}.
    Combined with Babai and Szemerédi's result for groups~\cite{BabaiS84}, this completes the proof.
  \item The pseudovariety $\pv{V}$ is \emph{permutative}, meaning that an identity $x_1 \cdots x_n \approx x_{\sigma(1)} \cdots x_{\sigma(n)}$ holds for all members of $\pv{V}$, where $\sigma \in \mathrm{Sym}(n)$ is some nontrivial permutation of the symbols $1, \dots, n$.
    This notion properly generalizes commutativity, which is characterized by the identity $xy \approx yx$.
    In \cref{sec:permutative} we present a direct proof that such pseudovarieties afford efficient compression, refining an earlier argument for commutative semigroups due to Fleischer~\cite{Fleischer19diss}.
    Our construction yields straight-line programs of essentially optimal length $\bigO(\log N)$ and width two. 
    (Matching lower bounds are established in \cref{sec:constant_length}.)

\end{itemize}

For pseudovarieties in the first class, our reductions are efficient in that questions about asymptotically optimal straight-line program length (and width) reduce to the group case.
The latter is discussed in \cref{sec:groups}, where we additionally present two new constructions for solvable groups, yielding straight-line programs of asymptotically optimal length $\bigO(\log N)$ but unbounded width, and of polylogarithmic length and bounded width, respectively.

\section{Preliminaries}

Due to the nature of our results, this work necessarily intersects several areas within the theory of semigroups and groups, as well as aspects of complexity theory.
As a consequence, we assume that the reader is already familiar with the foundational concepts in these fields or is prepared to consult the relevant literature for further background.

In this section, we provide a very brief overview of the necessary material from semigroup theory and complexity theory, along with a summary of the notational conventions used throughout this paper.
For background on group theory, we refer the reader to \cref{sec:groups}.

\subsection{Semigroups}

We assume that the reader is familiar with the theory of finite semigroups, and we refer to the excellent treatments of the subject by Almeida~\cite{Almeida1994}, and Rhodes and Steinberg~\cite{RhodesSteinberg2009} for relevant background material as well as any undefined terms.

Given a semigroup $S$, we write $T \leq S$ to indicate that $T \sse S$ is a subsemigroup. 
For an arbitrary subset $\Sigma \sse S$, we denote by $\gen[S]{\Sigma}$ the subsemigroup generated by $\Sigma$, consisting of all elements of $S$ expressible as a product of elements of $\Sigma$.
We write $\Sigma^{\leq k}$ for the set of all elements expressible in this way by a product of length at most $k$.
The set of completely regular elements of a finite semigroup $S$ is denoted by $I(S) = \{ s \in S : s^{\omega+1} = s\}$ where, as usual, $s^{\omega}$ denotes the unique idempotent power of an element $s$ of a finite semigroup.

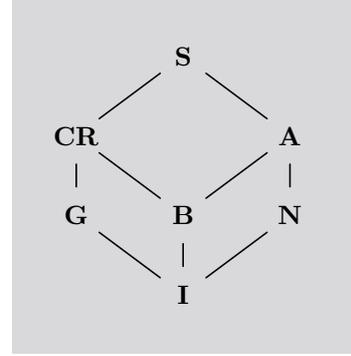
\begin{table}[h]
  \caption{Important pseudovarieties (left) and their relationships (right).}
  \label{tbl:pseudovarieties}
  \begin{minipage}{.66\textwidth}
  \begin{tabular}{ccl}
    \toprule
    Symbol & Identities & Description \\
    \midrule
    $\pv{S}$ &  ---  & all semigroups \\
    $\pv{B}$ & $x^2 \approx x$ & bands (idempotent semigroups) \\
    $\pv{I}$ & $x \approx y$ & trivial semigroups \\
    \midrule
    $\pv{CR}$ & $x^{\omega + 1} \approx x$ & completely regular semigroups \\
    $\pv{G}$ & $x^{\omega} \approx 1$ & groups \\
    \midrule
    $\pv{A}$ & $x^{\omega+1} \approx x^\omega$ & aperiodic semigroups \\
    $\pv{N}$ & $x^\omega \approx 0$ & nilpotent semigroups \\
    \midrule
    $\pv{Com}$ & $xy \approx yx$ & commutative semigroups \\
    \bottomrule
  \end{tabular}
  \end{minipage}
  \begin{minipage}{.33\textwidth}\hfill
  \begin{tikzpicture}
    \fill[color=lipicsLightGray] (-6.4em,-2.2em) -- (+6.4em,-2.2em) -- (+6.4em,11.2em) -- (-6.4em,11.2em) -- cycle;

    \node[circle] (S) at (0, 9em)     {$\mathclap{\pv{S}}$};
    \node[circle] (A) at (+4em, 6em)  {$\mathclap{\pv{A}}$};
    \node[circle] (R) at (-4em, 6em)  {$\mathclap{\pv{CR}}$};
    \node[circle] (N) at (+4em, 3em)  {$\mathclap{\pv{N}}$};
    \node[circle] (B) at ( 0em, 3em)  {$\mathclap{\pv{B}}$};
    \node[circle] (G) at (-4em, 3em)  {$\mathclap{\pv{G}}$};
    \node[circle] (I) at (0,0)        {$\mathclap{\pv{I}}$};

    \begin{scope}[line width=.6pt, line cap=round, shorten <=3pt, shorten >=3pt]
      \draw (S) -- (A);
      \draw (S) -- (R);
      \draw (A) -- (N);
      \draw (A) -- (B);
      \draw (R) -- (B);
      \draw (R) -- (G);
      \draw (N) -- (I);
      \draw (B) -- (I);
      \draw (G) -- (I);
    \end{scope}
  \end{tikzpicture}
  \end{minipage}
\end{table}

Most of our analysis will concern \emph{pseudovarieties} -- that is, classes of finite semigroups closed under formation of finite direct products, subsemigroups, and homomorphic images.
According to Reitermann~\cite[Theorem~3.1]{Reiterman1982}, such a class consists of all finite semigroups satisfying some set of profinite identities.
Pseudovarieties are also closely connected to classes of regular languages exhibiting natural closure properties, as established by Eilenberg~\cite{Eilenberg1976}.

Throughout, we use boldface type to denote pseudovarieties and specify defining sets of (profinite) identities using double-struck square brackets.
For example, $\pv{Com} = \pvI{xy \approx yx}$ indicates that $\pv{Com}$ is the pseudovariety consisting of all finite semigroups satisfying the identity $xy \approx yx$, that is, all finite commutative semigroups.
Some important pseudovarieties, which serve as convenient reference points, are listed in \cref{tbl:pseudovarieties}.

Central in this work, as they form primary obstructions, are the following pseudovarieties:
\begin{equation*}
  \pv{LRB} = \pvI{x^2 \approx x, xyx \approx xy},
  \quad 
  \pv{RRB} = \pvI{x^2 \approx x, xyx \approx yx}, 
  \quad
  \pv{T} = \pvI{x^2 \approx xyx \approx 0}.
\end{equation*}
For reference, $\pv{T}$ is a pseudovariety of nilpotent semigroups (that is, $\pv{T} \sse \pv{N}$), while the classes $\pv{LRB}$ and $\pv{RRB}$ consist of all finite \emph{left-regular} and \emph{right-regular bands}, respectively.

Other pseudovarieties of bands that we consider here are $\pv{RB} = \pvI{x^2 \approx x, xyz \approx xz}$, comprising \emph{rectangular bands};\footnote{Be aware that in the literature $\pv{RB}$ sometimes denotes the pseudovariety of \emph{regular bands} instead.} $\pv{Sl} = \pv{B} \cap \pv{Com} = \pvI{x^2 \approx x, xy \approx yx}$, comprising \emph{semilattices}; and $\pv{NB} = \pvI{x^2 \approx x, uxyv \approx uyxv}$, comprising \emph{normal bands}.
The (pseudo-)varieties of bands have been completely classified by Biryukov~\cite{Biryukov1970}, Fennemore~\cite{Fennemore1971}, and Gerhard~\cite{Gerhard1970}.

In addition, we consider pseudovarieties defined in terms of extensions by nilpotent semi\-groups, which are most conveniently expressed as Mal'cev products (though used here only for notation).
Given a pseudovariety~$\pv{V}$, we write $\pv{V} \pvM \pv{N}$ for the pseudovariety with $S \in \pv{V} \pvM \pv{N}$ if and only if the ideal $S^k \leq S$ belongs to $\pv{V}$ for some~$k \geq 1$~--~that is, the semigroup~$S$ is an extension of $S^k \in \pv{V}$ by the Rees quotient~$S / S^k \in \pv{N}$, which is obtained by identifying all elements of $S^k$.
Fixing $k \geq 1$ yields the pseudovariety~$\pv{V} \pvM \pv{N}_k$ instead, where $\pv{N}_k = \pvI{x_1 \cdots x_k \approx 0}$ refers to the pseudovariety of finite \emph{$k$-nilpotent} semigroups.

\subsection{Complexity}

We assume that the reader is familiar with standard complexity classes such as \LOGSPACE, \NL, \NC, \PTIME, \NP, and \PSPACE; see, for example, \cite{AroBar09}. 
Throughout, we write $\polylog n$ for $\log^{\bigO(1)}n$.

For sublinear time classes, we use random-access Turing machines meaning that the Turing machine has a separate address tape and a query state; whenever the Turing machine goes into the query state and the address tape contains the number $i$ in binary, the $i$th symbol of the input is read (the content of the address tape is \emph{not} deleted after that).
Apart from that, random-access Turing machines work like regular Turing machines.

For functions $t(n), s(n) \in \Omega(\log n)$, the classes $\DTISP(t(n), s(n))$ and $\NTISP(t(n), s(n))$ consist of the problems decidable by (non-)deterministic, $\mathcal{O}(t(n))$-time and $\mathcal{O}(s(n))$-space bounded, random-access Turing machines.
Be aware that there must be one Turing machine that simultaneously satisfies the time and space bound.
Without restricting the available space one obtains the classes $\DTIME(t(n))$ and $\NTIME(t(n))$.
We also define 
\[
  \NPOLYLOGTIME = \NTIME(\polylog n) = \bigcup_{\smash{c \geq 1}} \NTIME(\log^c n).
\]

The class $\ACz$ is defined as the class of problems decidable by polynomial-size, constant-depth Boolean circuits where all gates may have arbitrary fan-in.
The classes $\FOLL$ and $\qACz$ are defined analogously but allowing for circuits of polynomial size and depth $\bigO(\log \log n)$, and for quasipolynomial size (\ie $2^{\polylog n}$) and constant depth, respectively.
Throughout, we consider only uniform circuit classes (specifically, $\DTIME(\log n)$-uniform circuits for $\ACz$ and $\FOLL$, and $\DTIME(\polylog n)$-uniform circuits for $\qACz$), meaning that the circuits can be constructed (or verified) efficiently; see \cite{Vollmer99} for details.

\section{Compression via Straight-Line Programs}\label{sec:compression}

In this work, we are interested in the efficient representation of semigroup elements using straight-line programs.
These are commonly defined as circuits over the algebraic structure (e.g.\ \cite{Lohrey14compressed}), via context-free grammars that generate a single word (e.g.\ \cite{Lohrey2012survey}) or, in some cases, as sequences of elements corresponding to intermediate values of the former (e.g.\ \cite{BabaiS84}).

Here we take a pragmatic point of view.
Given a semigroup $S$, a \emph{straight-line program}~$\mathcal{A}$ over a set $\Sigma \sse S$ is a finite sequence of instructions to be executed in order (that is, without branches or loops) and operating on a potentially unbounded set of registers $\{ r_1, r_2, \dots \}$. 
Each register $r_k$ can store a single element of the semigroup $S$, and the straight-line program may use the following two instruction types to alter the contents of the registers.
\begin{samepage}
\smallskip
\begin{itemize}
  \addtolength{\itemsep}{.5pt}
  \newcommand{\instruction}[1]{\setlength\fboxsep{5pt}\colorbox{lipicsLightGray}{\makebox[10em][l]{\;$\vphantom{x}\smash{#1}$\hfill}}}
  \item Assign the fixed element $s \in \Sigma$ to the register $r_k$. 
    \hfill\instruction{r_k \gets s}
  \item Assign the product of the registers $r_i, r_j$ to the register $r_k$. 
    \hfill\instruction{r_k \gets r_i \cdot r_j}
\end{itemize}
\smallskip
\end{samepage}
In the second type of instruction, the registers $r_i$, $r_j$, and $r_k$ are not necessarily distinct, but we require that the input registers $r_i$ and $r_j$ were each assigned in some previous instruction.

The straight-line program $\mathcal{A}$ is said to \emph{compute} a semigroup element $t \in S$ if, upon completion of execution, some register $r_k$ contains the value $t$.
More generally, $\mathcal{A}$ computes some set $T \sse S$ if it computes every $t \in T$, and the largest such set is the \emph{value set} $V(\mathcal{A})$.

The \emph{length} $\ell({\mathcal{A}})$ and \emph{width} $w(\mathcal{A})$ of the straight-line program $\mathcal{A}$ are the number of its instructions and the number of registers it operates on, respectively.
Intuitively, the length and width of a straight-line program measure the time and space required to execute it. 

\medskip

Let $S$ be a finite semigroup and $\Sigma \subseteq S$.
Given a set $T \sse S$ and a width bound $2 \leq w \leq \infty$, we define the \emph{straight-line cost} of $T$ over $\Sigma$ to be the quantity 
\[
c^w_S(T; \Sigma) \coloneqq \min \bigg\{ \,\ell(\mathcal{A})\, : \parbox{16em}{\centering$\mathcal{A}$ is a straight-line program over $\Sigma$\\ with $T \sse V(\mathcal{A})$ and $w(\mathcal{A}) \leq w$}\bigg\}.
\]
The straight-line cost of an element $t \in S$, written $c^w_S(t; \Sigma)$, is defined analogously.
Sometimes we employ more general terminology and refer to either straight-line programs of \emph{bounded width}, corresponding to arbitrary $w < \infty$, or \emph{unbounded width}, in the case where $w = \infty$.

\medskip

As with conventional programs, straight-line programs can be composed sequentially and invoked as subroutines, modulo simple modifications such as register renaming. 
Applying such composition techniques leads to the following estimates.

\begin{lemma}\label{lem:augmented-generating-set}
  Let $S$ be a finite semigroup, $\Sigma, \Delta \sse S$, and $t \in S$.
  Then, for all $2 \leq w, \delta \leq \infty$,
  \vspace{-1ex}
  \[
    c^{w + \delta}_S(t; \Sigma) \leq c^{w}_S(t; \Sigma \cup \Delta) + c^{\delta}_S(\Delta; \Sigma)
    \quad\text{and}\quad
    c^{w+\delta - 1}_S(t; \Sigma) \leq c^{w}_S(t; \Sigma \cup \Delta) \cdot \max_{t' \in \Delta} c^{\delta}_S(t'; \Sigma).
  \]
\end{lemma}

A fundamental subroutine, which we often employ without explicit reference, is fast exponentiation via repeated squaring. 
In the context of straight-line programs, this technique dates back at least to 1937, when Scholz presented it as an application of \emph{addition chains} \cite{Scholz1937}.

\begin{observation}\label{obs:exponentiation}
  Let $S$ be a finite semigroup, $t \in S$, and $n \geq 1$.
  Then $c_S^2(t^n; \{t\}) \in \mathcal{O}(\log n)$.
\end{observation}

One of our primary concerns are worst-case bounds for $c^w_S(t; \Sigma)$ as $\Sigma \sse S$ and $t \in S$ vary; that is, in the quantity $C_S^w \coloneqq \max \big\{c^w_S(t; \Sigma) : t \in \gen[S]{\Sigma} \sse S \big\}$.
More specifically, we seek to asymptotically bound $C^w_S$ in the size of the semigroup $S$ as the latter ranges over the members of a given pseudovariety~$\pv{V}$.
To this end, let us define $C^w_{\pv{V}} \colon \mathbb{N} \to \mathbb{N}$ by
\begin{equation*}
  C^w_{\pv{V}}(N) \coloneqq \max \big\{ C^w_S : S \in \pv{V}, \abs{S} \leq N \big\}.
\end{equation*}

We then say that $\pv{V}$ \emph{admits} straight-line programs of width $w$ and length $\mathcal{O}(f(N))$ provided that $C^w_{\pv{V}}(N) \in \mathcal{O}(f(N))$ and, conversely, that $\pv{V}$ \emph{requires} straight-line programs of length $\Omega(f(N))$ provided that $C^\infty_{\pv{V}}(N) \in \Omega(f(N))$.
Finally, we say that a pseudovariety $\pv{V}$ affords \emph{efficient compression} (via straight-line programs) if $C^\infty_{\pv{V}}(N) \in \mathcal{O}(\polylog N)$.

\section{Obstructions to Efficient Compression}\label{sec:obstructions}

As mentioned in the introduction, the following three pseudovarieties constitute the primary obstructions to efficient compression in semigroups via straight-line programs:

\begin{equation*}
  \pv{LRB} = \pvI{x^2 \approx x, xyx \approx xy},
  \quad 
  \pv{RRB} = \pvI{x^2 \approx x, xyx \approx yx}, 
  \quad
  \pv{T} = \pvI{x^2 \approx xyx \approx 0}.
\end{equation*}

A simple consequence of their defining identities is that representing elements of their members in terms of generators essentially requires \emph{injective words}, that is, words in which each generator appears at most once.  
Indeed, for $\pv{T}$, all nonzero elements must be represented this way, while for $\pv{LRB}$ and $\pv{RRB}$, any word in generators defines the same element as the word obtained by keeping only the first or last occurrence of each generator, respectively.

Intuitively, since injective words contain no repeated factors, they cannot be efficiently compressed via straight-line programs (or any compression method relying on such repetitions).
However, because we measure efficiency relative to the size of the semigroup, a formal proof requires exhibiting members of these pseudovarieties in which some element cannot be expressed using few generators, as done by Fleischer~\cite[Lemma~4.11, Lemma~4.15]{Fleischer19diss}.

For the reader's convenience, we include Fleischer's construction below. 
Here, $\Omega_n(\pv{V})$ denotes the free pro-$\pv{V}$ semigroup on $n$ generators.
For $\pv{V} \in \{ \pv{LRB}, \pv{RRB}, \pv{T} \}$, this semigroup is finite (and hence a member of $\pv{V}$) and admits a convenient model consisting of all injective words over an $n$-element alphabet (and, in case of $\pv{T}$, an additional zero element) where composition is concatenation with subsequent application of the respective identities.

\begin{lemma}[Fleischer]\label{lem:obstructions}
  Let $\pv{V}$ be a pseudovariety that contains $\pv{LRB}$, $\pv{RRB}$, or $\pv{T}$.
  Then, for every $N \geq 1$, there exist $S \in \pv{V}$ with $\abs{S} \leq N$, a generating set $\Sigma \sse S$ with $\abs{\Sigma} \in \Omega(\sqrt{N})$, and an element $t \in S$ such that no subsemigroup generated by a proper subset of $\Sigma$ contains~$t$. 

  \smallskip

  \noindent{}In particular, $C^\infty_{\pv{V}}(N) \in \Omega(\sqrt{N})$; that is, $\pv{V}$ requires straight-line programs of length $\Omega(\sqrt{N})$.
\end{lemma}
\begin{proof}
  Let $s_1, \dots, s_n$ be the generators of $\Omega_n(\pv{LRB})$ and consider $S = \Omega_n(\pv{LRB}) / \theta$ where $\theta$ is the congruence generated by $s_i s_j \mathrel{\theta} s_j$ for all $i, j$ with $i+1 < j$.
  Note that, under the same conditions, we also have $s_j s_i = s_j s_i s_j \mathrel{\theta} s_j s_j = s_j$.
  Thus, the elements of $S$ are precisely those of the form $[s_i s_{i+1} \cdots s_j]_\theta$ with $1 \leq i \leq j \leq n$; in particular, it holds that $\abs{S} \in \mathcal{O}(n^2)$.

  The element $t = [s_1s_2 \cdots s_n]_\theta$, on the other hand, is \emph{not} a member of the subsemigroup generated by any proper subset of $\Sigma = \{ [s_1]_\theta, \dots, [s_n]_\theta\} \sse S$.
  Therefore, every straight-line program computing $t$ over $\Sigma$ must use every generator, and thus $c^\infty_S(t; \Sigma) \in \Omega(n)$.

  \smallskip

  The above argument proves the assertion in case $\pv{LRB} \sse \pv{V}$ and, by symmetry, this also holds in case $\pv{RRB} \sse \pv{V}$.
  In case $\pv{T} \sse \pv{V}$, one considers the Rees quotient $S = \Omega_n(\pv{T}) / I$ where $I$ is the ideal generated by all elements $s_is_j$ with $i+1 \neq j$. 
  Since the remainder of the argument is virtually identical, we omit the details.
\end{proof}

\section{Upper Bounds for Arbitrary Semigroups}

We complement Fleischer’s $\Omega(\sqrt{N})$ lower bound for all pseudovarieties $\pv{V}$ containing one of the obstructions $\pv{LRB}$, $\pv{RRB}$, or $\pv{T}$ by establishing a matching $\mathcal{O}(\sqrt{N})$ upper bound valid for all finite semigroups.
Our construction is a variant of the algorithm Re-Pair introduced by Larsson and Moffat \cite{LarssonMoffat2000}, which repeatedly compresses the most frequent digrams.

\begin{proposition}
  Let $\pv{S}$ be the pseudovariety of all finite semigroups. 
  Then $C^\infty_{\pv{S}}(N) \in \mathcal{O}(\sqrt{N})$.
\end{proposition}

\begin{proof}
  Let $S$ be a finite semigroup of size $\abs{S} \leq N$, and let $t \in S$.
  We show that for every generating set $\Sigma \sse S$, the straight-line cost $c^\infty_S(t; \Sigma)$ is bounded by $\mathcal{O}(\sqrt{N})$.
  To this end, we iteratively construct generating sets $\Sigma_0 \sse \Sigma_1 \sse \dots \sse S$, starting with $\Sigma_0 = \Sigma$, as follows.

  Given $\Sigma_k$, choose an arbitrary linear order on its elements, and let $w_k$ be the length-lexicographically minimal word over $\Sigma_k$ that evaluates to $t$ in $S$.
  If $w_k$ has length $n_k = 1$, then we terminate the process.
  Otherwise, among the digrams (that is, factors of length two) occurring in $w_k$, we choose one that occurs most frequently.
  We then replace a maximal collection of pairwise nonoverlapping occurrences of this digram by a letter $s_k \in S$ representing its value.
  The resulting word $\tilde w_k$ over $\Sigma_{k+1} \coloneqq \Sigma_k \cup \{s_k\}$ still evaluates to $t$, but is shorter.
  In particular, we have $\abs{w_{k+1}} \leq \abs{\tilde w_k} < \abs{w_k}$, so that $n_{k+1} < n_k$.

  Since the words $w_k$ are length-minimal, we have $n_k \leq N$: otherwise two prefixes of $w_k$ would evaluate to the same semigroup element, and deleting the factor between them would yield a shorter word representing $t$.
  Hence the process terminates after at most $N - 1$ steps, ending with a word $w_k$ over $\Sigma_k$ of length $n_k = 1$.
  For our purposes, however, the earlier stages of the process are more relevant, as digrams may occur with much higher frequencies, allowing their replacement to achieve greater compression.
  In fact, we will show below that there exists $K \in \mathcal{O}(\sqrt{N})$ such that $n_K \in \mathcal{O}(\sqrt{N})$.
  Note that this immediately implies 
  \[
    c^\infty_S(t; \Sigma) \leq c^\infty_S(t; \Sigma_K) + \sum_{k = 0}^{K-1} c^\infty_S(s_k; \Sigma_k)
    \leq \mathcal{O}(n_K) + \mathcal{O}(K) \in \mathcal{O}(\sqrt{N}).
  \] 

  Let us consider the word $w_k = a_1 \cdots a_{n_k}$ over $\Sigma_k$ at some intermediate step.
  Recall that the word $w_k$ is the length-lexicographically minimal representative of $t$ over $\Sigma_k$.
  Thus, for all indices $i,j$ with $1 \leq i < j \leq n_k$, the factor $w_k(i, j) \coloneqq a_i \cdots a_j$ is also the length-lexicographically minimal representative over $\Sigma_k$ of some element $t_k(i,j) \in S$.
  By uniqueness of such representatives, we have $w_{k}(i_1, j_1) = w_k(i_2, j_2)$ if and only if $t_k(i_1, j_1) = t_k(i_2, j_2)$.

  Let $I_k$ denote the set of all pairs $(i, j)$ with $1 \leq i < j \leq n_k$.
  Moreover, for $s \in S$, let us write $I_k(s) = \{ (i, j) \in I_k \mid t_k(i, j) = s\}$.
  Thus, $I_k$ is the disjoint union of the sets $I_k(s)$.
  Hence, we must have $\abs{I_k(s)} \geq \frac{1}{\abs{S}}\abs{I_k} \geq \frac{1}{N}\cdot\binom{n_k}{2}$ for some $s \in S$.
  Fix such an $s \in S$ and consider its length-lexicographically minimal representative $u$ over $\Sigma_k$.
  We must have $\abs{u} \geq 2$ because $\abs{I_k(s)} > 0$.
  Since $\abs{I_k(s)}$ is the number of times that $u$ occurs as a factor in $w_k$, and this number can only increase when replacing $u$ by one of its factors, we may assume that $u$ has length $\abs{u} = 2$.
  Thus, the value $s_k \in S$ of the most frequent digram in $w_k$ must also satisfy $\abs{I_k(s_k)} \geq \frac{1}{N}\binom{n_k}{2}$.
  A maximal collection of nonoverlapping occurrences of this digram then has size at least $\tfrac{1}{2}\abs{I_k(s_k)}$, so that we have $\abs{\tilde w_k} \leq \abs{w_k} - \frac{1}{2}\abs{I_k(s_k)}$.
  This yields
  \begin{equation}
    n_{k+1} \leq n_k - \frac{1}{2} \abs{I_k(s_k)} \leq n_k - \frac{1}{2N} \binom{n_k}{2} = n_k - \frac{n_k(n_k-1)}{4N}.
    \label{eqn:repair}
  \end{equation}
  The word lengths are thus majorized by a discrete logistic process, which in the regime relevant for us~--~where $n_k$ is still far from the stable value~$1$~--~decays harmonically.

  Indeed, writing $m_k \coloneqq n_k - 1$ and abbreviating $\alpha \coloneqq \tfrac{1}{4N}$, we obtain $m_{k+1} \leq m_k (1 - \alpha m_k)$ from \eqref{eqn:repair}.
  Passing to reciprocals and applying the inequality $\tfrac{1}{1-x} \geq 1 + x$, which is valid for all $x$ with $0 \leq x < 1$, then yields the estimate
  \[
    \frac{1}{m_{k+1}} \geq \frac{1}{m_k (1 - \alpha m_k)} \geq \frac{1 + \alpha m_k}{m_k} = \frac{1}{m_k} + \alpha.
  \]
  Iterating this estimate, we obtain
  \[
    \frac{1}{m_k} \geq \frac{1}{m_{k-1}} + \alpha \geq \frac{1}{m_{k-2}} + 2\alpha \geq \dotsb \geq \frac{1}{m_{0}} + k\alpha \geq k\alpha.
  \]
  Hence, $n_k = m_k + 1 \leq \tfrac{1}{k \alpha} + 1 = \tfrac{4N}{k} + 1 \leq 2\sqrt{N} + 1$ holds whenever $k \geq 2\sqrt{N}$.
  In particular, the smallest number $K$ with $n_K \leq 2\sqrt{N} + 1$ satisfies $K \leq 2\sqrt{N} + 1$.\footnote{As the $n_k$ are strictly decreasing, this also shows that our process stops after at most $4\sqrt{N} + 2$ rounds.}
\end{proof}

\section{Semigroups of Bounded Diameter}\label{sec:constant_length}

In this section, we classify pseudovarieties where every semigroup element can be written as a product of bounded length over any generating set (see \cref{thm:bounded-diameter}) and show that all other pseudovarieties require straight-line programs of length $\Omega(\log N)$.

\begin{definition}
  Let $\pv{V}$ be a pseudovariety.
  We say that $\pv{V}$ has \emph{bounded diameter} if there is a constant $D \geq 1$ such that, for every $S \in \pv{V}$ and generating set $\Sigma \sse S$, it holds that $\Sigma^{\leq D} = S$; that is, every element $t \in S$ is a product of generators from $\Sigma$ of length at most $D$.
\end{definition}

Clearly, if $\pv{V}$ has bounded diameter, then $C^2_{\pv{V}}(N) \leq 2D - 1$ (with $D$ as above); and, conversely, if $C^{\infty}_\pv{V}(N) \leq C$ for some constant $C$, then $\pv{V}$ has bounded diameter (with $D = 2^C$).
In other words, the pseudovarieties of bounded diameter are precisely the pseudovarieties admitting straight-line programs of length (and width) $\mathcal{O}(1)$.
An obvious example of such a pseudovariety is $\pv{N}_k = \pvI{x_1 \cdots x_k \approx 0}$, which consists of all finite $k$-nilpotent semigroups.

The following result, which will also be of use in the proof of our main theorem, generalizes this example to Mal'cev products of the form $\pv{V} \pvM \pv{N}_k$.
(Indeed, it applies to $\pv{N}_k = \pv{I} \pvM \pv{N}_k$ where $\pv{I} = \pvI{x \approx 1}$ is the trivial pseudovariety.)
Recall that a finite semigroup $S$ belongs to the pseudovariety $\pv{V} \pvM \pv{N}_k$ if and only if its ideal $S^k \leq S$ belongs to $\pv{V}$. 

\begin{lemma}\label{lem:nilpotent}
  Let $\pv{V} \sse \pv{V}' \pvM \pv{N}_k$ for a pseudovariety $\pv{V}'$ and $k \geq 1$.
  Then, for all $2 \leq w \leq \infty$,
  \[
    C^{w+1}_\pv{V}(N) \leq C^{w}_{\pv{V}'}(N) \cdot (4k - 3).
  \]
\end{lemma}

Herein, the factor $4k - 3$ is a constant and, hence, suppressed in Landau notation, in which the bound reads $C^{w+1}_{\pv{V}}(N) \in \mathcal{O}(C^{w}_{\pv{V}'}(N))$.
However, be aware that $\pv{N} = \bigcup_k \pv{N}_k$, which consists of all finite nilpotent semigroups, does \emph{not} afford efficient compression since $\pv{T} \sse \pv{N}$.

\begin{proof}
  Consider a semigroup $S \in \pv{V}$, a generating set $\Sigma \sse S$, and an element $t \in S$.
  If the element $t$ is not contained in the ideal $S^k \leq S$, then it can be written as a product of generators from $\Sigma$ of length less than $k$; hence $c^2_S(t; \Sigma) \leq 2k - 3$.
  Otherwise, since $S^k \in \pv{V}'$ is generated by the set $\Delta = \Sigma^{\leq 2k - 1} \cap S^k$, the second inequality of \cref{lem:augmented-generating-set} yields
  \[
    c^{w+1}_S(t; \Sigma) 
    \leq c^w_{S^k}(t; \Delta) \cdot \max_{t' \in \Delta}c^2_{S}(t'; \Sigma) 
    \leq C^w_{\pv{V}'}(\abs{S}) \cdot (4k - 3). \qedhere
  \]
\end{proof}

Another example of bounded diameter is given by $\pv{RB} = \pvI{x^2 \approx x, xyz \approx xz}$, which consists of all finite \emph{rectangular bands}, and satisfies $C^2_\pv{RB}(N) \leq 3$.
Combining the latter with \cref{lem:nilpotent} yields $C^3_{\pv{V}}(N) \leq 12k - 9$ for every pseudovariety $\pv{V} \sse \pv{RB} \pvM \pv{N}_k$.
Using a direct argument (which is straight-foreword and, therefore, omitted), this can be slightly improved.

\begin{lemma}\label{lem:rectangular-by-nilpotent}
  If $\pv{V} \sse \pv{RB} \pvM \pv{N}_k$ for some $k \geq 1$, then $C^2_{\pv{V}}(N) \leq 6k - 3$.
\end{lemma}

Fleischer showed that the pseudovariety $\pvL \pv{I}_k = \pvI{x_1 \cdots x_k y z_1 \cdots z_k \approx x_1 \cdots x_k z_1 \cdots z_k}$ has bounded diameter for all $k \geq 1$ \cite[Proposition~4.5]{Fleischer19diss}. 
However, since $\pvL \pv{I}_k \sse \pv{RB} \pvM \pv{N}_{2k} \sse \pvL \pv{I}_{3k}$ holds for every $k \geq 1$, this is equivalent to the above.
As it turns out, every pseudovariety of bounded diameter is contained in $\pv{RB} \pvM \pv{N}_k$ or, equivalently, in $\pvL \pv{I}_k$ for some $k \geq 1$.

\begin{theorem}\label{thm:bounded-diameter}
  Let $\pv{V}$ be a pseudovariety.
  Then exactly one of the following holds.
  \begin{itemize}
    \item The pseudovariety $\pv{V}$ has bounded diameter.
      In particular, $C^2_{\pv{V}}(N) \in \mathcal{O}(1)$.
    \item The pseudovariety $\pv{V}$ requires straight-line programs of length $\Omega(\log N)$.
  \end{itemize}
  Moreover, the former is the case if and only if $\pv{V} \subseteq \pv{RB} \pvM \pv{N}_k$ for some $k \geq 1$.
\end{theorem}

Before we prove \cref{thm:bounded-diameter} in its full generality, let us first establish the lower bound asserted by the second alternative.
We will distinguish between two cases, depending on how the pseudovariety in question is situated relative to pseudovarieties $\pvL\pv{I}$ and $\pv{U}$.

The first case concerns the pseudovariety $\pvL\pv{I} = \pvI{x^\omega y x^\omega \approx x^\omega}$ of \emph{locally trivial} semigroups, that is, semigroups having only trivial submonoids.
By \cite[Exercise~6.4.2]{Almeida1994}, it satisfies
\[
  \pvL\pv{I} = \pv{RB} \pvM \pv{N} = \bigcup_{k \geq 1} \pv{RB}\pvM\pv{N}_k = \bigcup_{k \geq 1} \pvL\pv{I}_k.
\]

The following observation is also contained in Fleischer's dissertation~\cite[Proposition~4.7]{Fleischer19diss}.

\begin{lemma}[Fleischer]\label{lem:locally-trivial}
  Let $\pv{V}$ be a pseudovariety with $\pv{V} \not\sse \pvL\pv{I}$. 
  Then $C^\infty_{\pv{V}}(N) \in \Omega(\log N)$.
\end{lemma}
\begin{proof}
  Let $M$ be a nontrivial monoid with $M \in \pv{V}$, which exists since $\pv{V} \not\sse \pvL\pv{I}$.
  Let $e \in M$ be its neutral element, and let $s \in M$ with $s \neq e$.
  We consider the subsemigroup $S$ of the Cartesian power $M^{\times n}$ that is generated by $\Sigma = \{s_1, \dots, s_n\}$ where $s_i$ has $i$th coordinate $s$ and all other coordinates equal to $e$.
  Clearly, $\log \abs{S} \in \Theta(n)$.
  Since the element $t = s_1 \cdots s_n \in S$ is not contained in the subsemigroup generated by any proper subset of $\Sigma$, every generator must appear in every straight-line program computing $t$ over $\Sigma$; hence, $c^\infty_S(t; \Sigma) \in \Omega(\log \abs{S})$.
\end{proof}

The second case concerns the pseudovariety $\pv{U} = \pv{T} \cap \pv{Com} = \pvI{xy \approx yx, x^2 \approx 0}$, which serves as a direct obstruction to the property of having bounded diameter.

\begin{lemma}\label{lem:tcom}
  Let $\pv{V}$ be a pseudovariety with $\pv{U} \sse \pv{V}$. 
  Then $C^\infty_{\pv{V}}(N) \in \Omega(\log N)$.
\end{lemma}
\begin{proof}
  For every $n \geq 1$, the semigroup $S = \gen{ s_1, \dots, s_n : s_is_j = s_js_i,\, s_i^2s_j = s_i^2 \; (1 \leq i,j \leq k)}$ is contained in $\pv{U} \sse \pv{V}$.
  Indeed, the defining relations imply $\smash{s_i^2 = s_i^2s_j^2 = s_j^2}$, so that $s^2$ is the zero element of $S$ for every $s \in S$.
  The semigroup $S$ contains exactly $2^n$ elements, namely the zero element and every product of distinct elements from $\Sigma \coloneqq \{ s_1, \dots, s_n \}$.

  Clearly, the element $t = s_1 \cdots s_n \in S$ is not contained in the subsemigroup generated by any proper subset of $\Sigma$.
  As in the proof of \cref{lem:locally-trivial}, this implies $c^\infty_S(t;\Sigma) \in \Omega(\log \abs{S})$.
\end{proof}

\begin{proof}[Proof of \cref{thm:bounded-diameter}]
  In view of \cref{lem:rectangular-by-nilpotent} as well as \cref{lem:locally-trivial,lem:tcom}, it suffices to show that if $\pv{V} \sse \pvL\pv{I}$ and $\pv{U} \not\sse \pv{V}$, then $\pv{V} \sse \pv{RB}\pvM\pv{N}_k$ for some $k \geq 1$.
  To this end, let us first note that $\pvL\pv{I} = \pv{RB} \pvM \pv{N}$; see \cite[Exercise~6.4.2]{Almeida1994}.
  Assuming $\pv{V} \sse \pvL\pv{I}$, we therefore have that $\pv{V} \sse \pv{RB}\pvM\pv{N}_k$ if and only if $\pv{V} \cap \pv{N} \sse \pv{N}_k$.
  Almeida and Reilly \cite[Proposition~4.4]{AlmeidaReilly1984} have shown that the latter is equivalent to the condition $\pv{U}_{k+1} \not\sse \pv{V} \cap \pv{N}$ where $\pv{U}_{k+1} = \pv{U} \cap \pv{N}_{k+1}$.
  Since $\bigcup_{k \geq 1} \pv{U}_k = \pv{U} \sse \pv{N}$, this condition holds for some $k \geq 1$ provided that $\pv{U} \not\sse \pv{V}$.
\end{proof}

\section{Permutative Semigroups}\label{sec:permutative}

In this section, we examine efficient compression in pseudovarieties where every member satisfies a \emph{common} permutation identity -- that is, an identity $x_1 \cdots x_n \approx x_{\sigma(1)} \cdots x_{\sigma(n)}$ where $\sigma \in \mathrm{Sym}(n)$ is a nontrivial permutation of the symbols $1, \dots, n$.
We refer to such pseudovarieties, and their members, as \emph{permutative}.\footnote{Be aware that some authors define a permutative pseudovariety as a pseudovariety consisting of permutative semigroups, meaning that the permutation identity may differ between its members.}
This notion generalizes commutativity, which corresponds to the identity $xy \approx yx$, and is equivalent to it for monoids.

\begin{proposition}\label{pro:permutative}
  Let $\pv{V}$ be a permutative pseudovariety.
  Then $C_{\pv{V}}^2(N) \in \mathcal{O}(\log N)$.
\end{proposition}

This improves upon an upper bound previously established by Fleischer~\cite[Lemma~4.9]{Fleischer19diss}, which shows that commutative semigroups admit straight-line programs of width three and length $\mathcal{O}(\smash{\log}^2 N)$.
In view of \cref{thm:bounded-diameter}, our bound is asymptotically optimal.

\begin{proof}
  According to Putcha and Yaqub~\cite[Theorem~1]{PutchaYaqub1971}, every permutation identity implies the identity $u_1 \cdots u_k xy v_1 \cdots v_k \approx u_1 \cdots u_k yx v_1 \cdots v_k$ for each sufficiently large $k \geq 0$.
  In the following, we fix some $k \geq 1$ such that $\pv{V}$ satisfies this identity.

  \smallskip

  Consider a semigroup $S \in \pv{V}$, a generating set $\Sigma \subseteq S$, and some element $t \in S$.
  Without loss of generality, the element $t$ is not a product of generators with $2k$ or fewer factors.
  Under these assumptions, the element $t$ can then be written as
  \begin{equation}\label{eqn:permutative_normal-form}
    t = u \, s_1^{\nu_1} \cdots\, s_n^{\nu_n} \, v
    \qquad
    \text{with } s_1, \dots, s_n \in \Sigma \text{, } u, v \in \Sigma^k \text{, and } \nu_1, \dots, \nu_n \geq 0.
  \end{equation}

  Fix the elements $s_1, \dots, s_n \in \Sigma$ and $u,v \in \Sigma^k$ involved in such an expression.
  Once these are fixed, we associate to each $\mu = (\mu_1, \dots, \mu_n) \in \mathbb{N}^n$ the element $s^\mu \coloneqq s_1^{\mu_1} \cdots\, s_n^{\mu_n} \in S \cup \{1\}$.
  We then also fix the \emph{lexicographically} minimal tuple of exponents $\nu \in \mathbb{N}^n$ such that $t = u\, s^\nu \, v$.

  Suppose given $\xi, \eta \in \mathbb{N}^n$ such that $\xi, \eta \leq \nu$ in the \emph{component-wise} order.
  We claim that, under these conditions, $u\, s^\xi = u\, s^\eta$ implies $\xi = \eta$.
  Indeed, if $\xi \leq_{\text{lex}} \eta$, then
  \[
    t = u\, s^\nu \, v = u\, s^{\eta} s^{\nu - \eta}\,v = u\, s^{\xi} s^{\nu - \eta}\, v = u\, s^{\nu - \eta + \xi}\, v
  \]
  and, clearly, $\nu - \eta + \xi \leq_{\text{lex}} \nu$; hence, $\xi = \eta$ by minimality of $\nu$.
  This establishes our claim, which implies $(\nu_1 + 1) \cdot \ldots \cdot (\nu_n + 1) \leq \abs{S}$ and, hence, $\log(\nu_1 + 1) + \ldots + \log(\nu_n + 1) \leq \log \abs{S}$.
  
  \smallskip

  Computing $t$ as in the expression~\eqref{eqn:permutative_normal-form} and using fast exponentiation yields a straight-line program of width three and length $\mathcal{O}(\log \abs{S})$ -- the latter being due to the above inequality.

  Limited to two registers, we instead proceed as follows.
  First, we use a simultaneous variant of fast exponentiation to compute an element $\tilde s \in S$ abstractly equivalent to $s^{\nu_1}_1 \cdots\, s^{\nu_n}_n$ under commutation.
  More precisely, let $r = \max_i \lceil\log (\nu_i + 1)\rceil$, and $\tilde s_1, \dotsc, \tilde s_r \in S \cup \{1\}$ with $\tilde s_i$ being the product (in some order) of those $s_j$ where the $i$th digit in the binary representation of $\nu_j$ equals $1$.
  Clearly, $\tilde s = (((\tilde s_r)^2 \cdot \tilde s_{r-1})^2 \cdot \ldots \cdot \tilde s_2 )^2 \cdot \tilde s_1 \in S$ can then be computed, as such, within $\mathcal{O}(\log \abs{S})$ instructions using only two registers.
  From there, we can obtain the element $t = u\, s_1^{\nu_1} \cdots\, s_n^{\nu_n} \, v = u \, \tilde s \, v$ using an additional $4k \in \mathcal{O}(1)$ instructions.
\end{proof}

Combining \cref{lem:obstructions} and \cref{pro:permutative}, we arrive at the following characterization for pseudovarieties of aperiodic semigroups, i.e., for subpseudovarieties of $\pv{A} = \pvI{x^{\omega+1} \approx x^{\omega}}$.

\begin{theorem}\label{thm:main-aperiodic}
  Let $\pv{V} \subseteq \pv{A}$ be a pseudovariety.
  The following are equivalent.
  \begin{bracketenumerate}
    \item The pseudovariety $\pv{V}$ affords efficient compression.
    \item The pseudovariety $\pv{V}$ contains neither $\pv{LRB}$, $\pv{RRB}$, nor $\pv{T}$.
    \item The pseudovariety $\pv{V}$ is permutative.
    \item The pseudovariety $\pv{V}$ admits straight-line programs of width two and length $\mathcal{O}(\log N)$.
  \end{bracketenumerate}
\end{theorem}
\begin{proof}
  See \cref{lem:obstructions} for $(1) \Rightarrow (2)$, \Cref{pro:permutative} entails $(3) \Rightarrow (4)$, and $(4) \Rightarrow (1)$ holds by definition.
  For the implication $(2) \Rightarrow (3)$ we refer to \cite[Corollary~16]{Thumm2025}.
\end{proof}

\section{Groups}\label{sec:groups}

Having just established our main theorem for the case of aperiodic semigroups (\cref{thm:main-aperiodic}), we now turn our attention to the pseudovarieties of groups.

\medskip

Most of our group theory notation follows Robinson's book \cite{Robinson96book}.
In particular, $N\trianglelefteq G$ denotes that $N$ is a normal subgroup of a group $G$. Furthermore, $g^h = h^{-1}gh$ denotes the conjugate of $g\in G$ by $h\in G$ and, for $\Delta, \Sigma \sse G$, we write $[\Delta, \Sigma]$ for the subgroup generated by all commutators $[g,h] = g^{-1}h^{-1}gh$ with $g \in \Delta$ and $h \in \Sigma$.
We differ from the notation in Robinson's book, however, by writing $\ncl[G]{\Sigma}$ to denote the normal subgroup generated by~$\Sigma$; that is, $\ncl[G]{\Sigma} = \gen[G]{ g^h : g \in \Sigma, h \in G}$.
Note that, since we only consider finite groups, the subgroup generated by a (nonempty) set $\Sigma$ coincides with the subsemigroup generated by it.

By Lagrange's theorem, we have the following well-known and straight-forward observation, which will be used in the following without further reference.
\begin{observation}\label{obs:log_gen_set}
  If $H \leq G$ with $H \neq G$, then $2 \cdot \abs{H} \leq \abs{G}$.
  In particular, for every $\Sigma \sse G$ there exists a subset $\Delta \sse \Sigma$ such that $\gen[G]{\Delta} = \gen[G]{\Sigma}$ and $\abs{\Delta} \leq \log \abs{\gen[G]{\Sigma}}$.
\end{observation}

\subsection{General Groups}

When considering a group $G$, we allow for the formation inverses in straight-line programs; that is, a \emph{group straight-line program} over $\Sigma \sse G$ may use the following instruction types.
\begin{samepage}
\smallskip
\begin{itemize}
  \addtolength{\itemsep}{.5pt}
  \newcommand{\instruction}[1]{\setlength\fboxsep{5pt}\colorbox{lipicsLightGray}{\makebox[10em][l]{\;$\vphantom{x}\smash{#1}$\hfill}}}
  \item Assign the fixed element $g \in \Sigma$ to the register $r_k$. 
    \hfill\instruction{r_k \gets g}
  \item Assign the product of the registers $r_i, r_j$ to the register $r_k$. 
    \hfill\instruction{r_k \gets r_i \cdot r_j}
  \item Assign the inverse of the registers $r_i$ to the register $r_k$. 
    \hfill\instruction{r_k \gets r_i^{-1}}
\end{itemize}
\smallskip
\end{samepage}

This modification does not increase the expressiveness for finite groups compared to ordinary (semigroup) straight-line programs, since one can compute the inverse of a group element $g$ as $g^{-1} = g^{\omega - 1}$.
The following observation shows that the decrease in cost afforded by the extra instruction type is often negligible as well.
Therein, and throughout this section, the \emph{group straight-line costs}, which are defined just as in \cref{sec:compression}, are marked with a tilde.

\begin{lemma}\label{lem:inverses_SLP}
  Let $\pv{H} \sse \pv{G}$ be a pseudovariety of groups.
  Then, for all $2 \leq w \leq \infty$, 
  \vspace{-1ex}
  \[
    C^\infty_{\pv{H}}(N) \leq 2 \cdot \tilde C^\infty_{\pv{H}}(N) + \mathcal{O}(\log N)
    \quad\text{and}\quad
    C^{w+1}_{\pv{H}}(N) \leq \tilde C^w_{\pv{H}}(N) \cdot \mathcal{O}(\log N).
  \]
\end{lemma}

\begin{proof}
  The second inequality arises from the on-the-fly computation of inverses via fast exponentiation.
  It is a consequence of the second inequality in \cref{lem:augmented-generating-set} and \cref{obs:exponentiation}.

  To see the first inequality, let us fix a group $G \in \pv{H}$, a generating set $\Sigma \sse G$, and an element $t \in G$.
  Let $\tilde{\mathcal{A}}$ be a group straight-line program computing $t$ over $\Sigma$ operating on the set of register $R = \{r_1, r_2, \dots\}$.
  Further, let $\bar R = \{ \bar r_1, \bar r_2, \dots \}$ be a disjoint copy of $R$.
  We modify the straight-line program $\tilde{\mathcal{A}}$ as follows.
  For each instruction $r_k \gets g$ or $r_k \gets r_i \cdot r_j$, we add an additional instruction $\bar r_k \gets g^{-1}$ or $\bar r_k \gets \bar r_j \cdot \bar r_i$ immediately after it.
  For each instruction $r_k \gets r_i^{-1}$, we would now like to either assign $r_k \gets \bar r_i$ and $\bar r_k \gets r_i$, or swap the registers $r_k$ and $\bar r_k$ in case $k = i$.
  While such instructions are not permitted in our definition of straight-line programs, they can be emulated by renaming registers in all subsequent instructions.
  Once this is done, we obtain an ordinary straight-line program $\mathcal{A}$ computing $t$ over $\Sigma \cup \Sigma^{-1}$ and operating on $R \cup \bar{R}$.
  It satisfies $w(\mathcal{A}) = 2 \cdot w(\tilde{\mathcal{A}})$ and $\ell(\mathcal{A}) \leq 2 \cdot \ell(\tilde{\mathcal{A}})$.

  The above shows that $c^{2w}_G(t; \Sigma \cup \Sigma^{-1}) \leq 2 \cdot \tilde c^w_G(t; \Sigma)$.
  In light of this inequality, it suffices to show that the set $\Sigma^{-1}$ can be computed efficiently from $\Sigma$.
  Since we did not assert bounds for specific generating sets, but only for $C^\infty_G$, we may assume, without loss of generality, that $\Sigma$ is an inclusion-minimal generating set.
  Then $\abs{\Sigma} \leq \log \abs{G}$ holds by Lagrange's theorem; see \cref{obs:log_gen_set}.
  We claim that, under these conditions, $c^\infty_G(\Sigma^{-1}; \Sigma) \in \mathcal{O}(\log \abs{G})$.

  Let $\Sigma = \{g_1, \dotsc, g_n\}$.
  A straight-line program computing $\Sigma^{-1}$ can proceed as follows.
  First, produce every generator $g_i$ (by assigning them to distinct registers).
  Then compute the successive products $h_i = g_1 \cdots g_i$ and $k_i = g_i \cdots g_n$ for all $1 \leq i \leq n$.
  Next, form the inverse $\bar g = (g_1 \cdots g_n)^{-1}$ of $h_n = k_n$ using fast exponentiation.
  Finally, extract the inverses of the generators $g_i$ as $g^{-1}_i = k_{i+1} \cdot \bar g \cdot h_{i-1}$ for $1 \leq i \leq n$ (omitting the factors $h_0$ and $k_{n+1}$ for $i = 1$ and $i = n$, respectively).
  Following this strategy, we obtain the required bound.
\end{proof}

Due to the above, the following well-known result by Babai and Szemer\'edi~\cite[Theorem~3.1]{BabaiS84}, which was originally stated for groups, can also be utilized in the semigroup setting.

\begin{rlemma}[Babai, Szemer\'edi]
  The pseudovariety $\pv{G} = \pvI{x^\omega \approx 1}$, which consists of all finite groups, admits straight-line programs of width\footnote{While Babai and Szemer\'edi \cite{BabaiS84} do not discuss width bounds, this follows easily from their proof.} $\mathcal{O}(\log N)$ and length $\mathcal{O}(\smash{\log}^2 N)$.
\end{rlemma}

Note that the bounds presented by Babai and Szemer\'edi are not necessarily optimal.
For instance, \cref{pro:permutative} implies that the pseudovariety $\GAb$, which consists of all finite Abelian groups, satisfies $C^2_{\GAb}(N) \in \mathcal{O}(\log N)$.
In the following, we present some techniques that might aid in future improvements of the bounds for general groups.
The techniques will then be applied in \cref{sub:solvable-groups} to obtain such improvements for solvable groups. 

\begin{definition}
  Let $G$ be a group.
  We call a generating set $\Sigma \sse G$ \emph{adapted} to a subnormal series\footnote{Recall that $G = G_0 \geq \dots \geq G_{n} = 1$ is a subnormal series if $G_{i}$ is normal in $G_{i-1}$ for all $1 \leq i \leq n$.} $G = G_0 \trianglerighteq \dots \trianglerighteq G_{n} = 1$ provided that $\Sigma \cap G_{i-1}$ generates $G_{i-1}$ for all $1 \leq i \leq n$.
\end{definition}

Moreover, we say that $\Sigma \sse G$ is a \emph{polycyclic} generating set if it is adapted to a subnormal series $G = G_0 \trianglerighteq \dots \trianglerighteq G_{n} = 1$ with $G_{i-1} / G_i$ cyclic for all $1 \leq i \leq n$.
Note that, as every generating set of an Abelian group is polycyclic, it follows that $\Sigma \sse G$ is a polycyclic generating set if and only if it is adapted to a subnormal series with Abelian quotients.

If a generating set $\Sigma \sse G$ is adapted to a subnormal series, then the associated straight-line costs $c^w_G(t;\Sigma)$ are essentially dominated by those in the consecutive quotients.

\begin{lemma}\label{lem:subnormal_SLP}
  Let $G$ be a group generated by a set $\Sigma \sse G$, and let $\pv{H} \sse \pv{G}$.\footnote{Here, we allow $\pv{H} \sse \pv{G}$ to be an \emph{arbitrary} class of finite groups.}
  Furthermore, suppose that $\Sigma$ is adapted to a subnormal series $G = G_0 \trianglerighteq \dots \trianglerighteq G_{n} = 1$ with $H_{i} \coloneqq G_{i-1} / G_{i}$ satisfying $H_i \in \pv{H}$ for all $1 \leq i \leq n$.
  Then, for all $t \in G$ and $2 \leq w \leq \infty$,
  \[
    c^{w+1}_G(t; \Sigma) \leq C^w_{\pv{H}}(\abs{H_1}) + \ldots + C^w_{\pv{H}}(\abs{H_{n}}) + n - 1.
  \]
  In particular, if $C^w_{\pv{H}}(N) \in \mathcal{O}(\log^c N)$ for some $c \geq 1$, then $c^{w+1}_G(t; \Sigma) \in \mathcal{O}(\log^c \abs{G})$.
\end{lemma}

Note that this yields yet another proof that the pseudovariety $\GAb$ of finite Abelian groups admits straight-line programs of bounded width and logarithmic length.
Indeed, we can simply take $\pv{H}$ to be the \emph{class} of all finite cyclic groups, which satisfies $C^2_{\pv{H}}(N) \in \mathcal{O}(\log N)$.

\begin{proof}
  In the additional register, we accumulate $t = t_1 \cdots t_n$ with $t_i \in G_{i-1}$.
  Specifically, suppose that $t'_i = (t_1  \cdots t_{i-1})^{-1} t \in G_{i-1}$ is already determined.
  Lifting an appropriately chosen straight-line program computing $\pi_i(t'_i) \in H_i$ over $\pi_i(\Sigma \cap G_{i-1})$ -- where $\pi_i \colon G_{i-1} \to H_i$ denotes the quotient homomorphism -- yields a straight-line program computing $t_i \in G_{i-1}$ over~$\Sigma \cap G_{i-1}$ with $\pi_i(t_i) = \pi_i(t'_i)$.
  Since the latter implies $t'_{i+1} \in G_{i}$, we may continue the process until we eventually arrive at $t'_{n+1} \in G_{n} = 1$; hence, $t = t_1 \cdots t_n$.

  For the addendum we assume, without loss of generality, that $H_i \neq 1$ for all $1 \leq i \leq n$, so that $\abs{H_1}, \dots, \abs{H_n} \geq 2$.
  Then, since $\abs{H_1} \cdot \ldots \cdot \abs{H_n} = \abs{G}$ by Lagrange's theorem, we obtain the equality $\log \abs{H_1} + \ldots + \log \abs{H_n} = \log \abs{G}$ and the bound $n \leq \log \abs{G}$.
  This implies desired bound on $c^{w+1}_G(t; \Sigma)$ by superaditivity of the function $\log^c \colon [1, \infty) \to [0, \infty)$.
\end{proof}

\subsection{Solvable Groups}\label{sub:solvable-groups}

Let us now turn our attention to \emph{solvable} groups; that is, to groups admitting a subnormal series with Abelian consecutive quotients.
The finite such groups form the pseudovariety~$\Gsol$.

\begin{theorem}\label{thm:solvable-groups}
  Let $\pv{H}$ be a pseudovariety with $\pv{H} \sse \Gsol$.
  Then
  \vspace{-1ex}
  \[
    C^\infty_{\pv{H}}(N) \in \mathcal{O}(\log N)
    \quad\text{and}\quad
    C^4_{\pv{H}}(N) \in \mathcal{O}(\log^3 N).
  \]
\end{theorem}

We suspect that it is also possible to show that $C_{\pv{H}}^w(N) \in \mathcal{O}(\log^2 N)$ for some slightly larger width bound $w < \infty$ using similar techniques. 
On the other hand, to answer whether or not the bound $C_{\pv{H}}^w(N) \in \mathcal{O}(\log N)$ holds for some $w < \infty$ is likely to require new ideas.

\medskip

In the following, we prove \cref{thm:solvable-groups} by providing efficient constructions of polycyclic generating sets for solvable groups -- the unbounded and bounded width cases being treated separately.
Throughout, we write $G = G^{(0)} \trianglerighteq G^{(1)} \trianglerighteq \dots $ for the \emph{derived series} of a group~$G$ where $G^{(k)} = [G^{(k-1)},G^{(k-1)}]$ and abbreviate $G' \coloneqq G^{(1)}$ and $G'' \coloneqq G^{(2)}$.

\begin{lemma}\label{lem:derived-series}
  Let $G$ be a group, $N \trianglelefteq G$ a normal subgroup, and $\Delta \sse N$.
  Then
  \[
    N = \gen[G]{\Delta} N'' \implies N' = \ncl[G]{[\Delta, \Delta]} N''
    \quad\text{and}\quad
    N = \gen[G]{\Delta} N' \implies N = \ncl[G]{\Delta} N''.
  \]
  In particular, if $N$ is a solvable group, then $N = \gen[G]{\Delta} N' \implies N = \ncl[G]{\Delta}$.
\end{lemma}
\begin{proof}
  The first implication can be found in Robinson's book \cite[5.1.7]{Robinson96book}.
  For the second implication, note that if $N = \gen[G]{\Delta} N'$, then $N = \gen[G]{\Delta} \, [\gen[G]{\Delta} N', \gen[G]{\Delta} N']$ and, hence,
  \[
		N \sse \ncl[G]{\Delta} \, [\ncl[G]{\Delta} N', \ncl[G]{\Delta} N']
		= \ncl[G]{\Delta} \, [\ncl[G]{\Delta}, \ncl[G]{\Delta}] \, [\ncl[G]{\Delta}, N']\, [N',N'] 
		= \ncl[G]{\Delta} N''
  \]
  where the second equality is due to well-known commutator identities; see \cite[5.1.5]{Robinson96book}.

  The addendum follows by induction on the derived length of $N$.
  We omit the details.
\end{proof}

Given subsets $\Delta, \Sigma \sse G$ and $k \geq 0$, we write $\normalop{\Delta}{\Sigma}{k} = \{ g^h : g \in \Delta, h \in \Sigma^{\leq k}\}$ and, in case $\Delta$ consists of a single element $g \in G$, we abbreviate this to $\normalop{g}{\Sigma}{k}$.
This construction can be used to efficiently construct a generating set of the normal closure.

\begin{lemma}\label{lem:normal-closure}
  Let $G$ be a group, $\Sigma \sse G$ a generating set, and $\Delta \sse G$.
  If $k \geq \log \abs{\ncl[G]{\Delta} / \gen[G]{\Delta}}$, then there exists some $\Xi \sse \normalop{\Delta}{\Sigma}{k}$ with $\ncl[G]{\Delta} = \gen[G]{\Delta \cup \Xi}$ and $\tilde c^\infty_G(\Xi; \Delta \cup \Sigma) \leq 5k$.
\end{lemma}
\begin{proof}
  To an initially empty set $\Xi$, we successively add elements $g^h$ with $g\in \Delta \cup \Xi$, $h \in \Sigma$, and such that (the size of) the group $\gen[G]{\Delta \cup \Xi}$ increases in every step.
  Since $\Delta \cup \Xi \sse \ncl[G]{\Delta}$, this process terminates after at most $\log \abs{\ncl[G]{\Delta} / \gen[G]{\Delta}} \leq k$ steps.
  Hence, $\Xi \sse \normalop{\Delta}{\Sigma}{k}$ and, clearly, $\tilde c^\infty_G(\Xi; \Delta \cup \Sigma) \leq 5k$.
  We show that $\ncl[G]{\Delta} = \gen[G]{\Delta \cup \Xi}$ holds eventually.

  Suppose that $g^h \not\in \gen[G]{\Delta \cup \Xi}$ for some $g \in \gen[G]{\Delta \cup \Xi}$ and $h \in G$.
  Since $\Sigma$ generates $G$ as a semigroup, we may assume that $h \in \Sigma$.
  Moreover, writing $g = g_1 \cdots g_n$ with $g_1, \dots, g_n \in \Delta \cup \Xi$, we see that $g_i^h \not\in \gen[G]{\Delta \cup \Xi}$ for some $1 \leq i \leq n$, for otherwise $g^h = g_1^h \cdots g_n^h \in \gen[G]{\Delta \cup \Xi}$.
  It follows that we may still add the element $g_i^h$ to the set $\Xi$, thereby continuing the process.
\end{proof}

Combining the above \cref{lem:derived-series,lem:normal-closure}, we arrive at the following result. Together with \cref{lem:subnormal_SLP}, it implies the assertion of \cref{thm:solvable-groups} for unbounded width.

\begin{proposition}\label{pro:derived-series}
  Let $G$ be a solvable group with a generating set $\Sigma$. Then there exists a generating set $\Delta \sse G$ adapted to the derived series of $G$ with $\tilde c^\infty_G(\Delta; \Sigma) \in \mathcal{O}(\log \abs{G})$.
\end{proposition}
\begin{proof}
  For every $i \geq 0$, we will construct a set $\Delta_i \sse G^{(i)}$ with $G^{(i)} = \gen[G]{\Delta_i} G^{(i+1)}$ and such that $\tilde c^\infty_G(\Delta_0; \Sigma) \leq \log {\lvert G^{(0)} / G^{(1)} \rvert}$ and $\tilde c^\infty_G(\Delta_{i}; \Delta_{i-1} \cup \Sigma) \leq 17 \cdot \log {\lvert G^{(i-1)} / G^{(i+1)} \rvert}$ for $i \geq 1$.
  In particular, the union $\Delta = \bigcup_{i \geq 0} \Delta_i$ will have the required properties, as 
  \[
    \tilde c^\infty_G(\Delta; \Sigma) \leq \tilde c^\infty_G(\Delta_0; \Sigma) + \sum_{\smash{i \geq 1}} \tilde c^\infty_G(\Delta_i; \Delta_{i-1} \cup \Sigma) \leq 18 \cdot \log \abs{G} + 17 \cdot \log \abs{G'} \in \mathcal{O}(\log \abs{G}).
  \]

  To begin the construction, we choose any $\Delta_0 \sse \Sigma$ of size $\abs{\Delta_0} \leq \log {\lvert G^{(0)} / G^{(1)} \rvert}$ with the property that $G = \gen[G]{\Delta_0} G'$.
  Now suppose that we have already constructed $\Delta_0, \dots, \Delta_{i-1}$.
  To construct $\Delta_i$ for $i \geq 1$, we proceed as follows.
  First, using \cref{lem:normal-closure}, choose $\Xi \sse G^{(i-1)}$ such that $\ncl[G]{\Delta_{i-1}} G^{(i+1)} = \gen[G]{\Delta_{i-1} \cup \Xi} G^{(i+1)}$ and $\tilde c^\infty_G(\Xi; \Delta_{i-1} \cup \Sigma) \leq 5 \cdot \log {\lvert G^{(i-1)} / G^{(i+1)} \rvert}$.
  We then have $G^{(i-1)} = \ncl[G]{\Delta_{i-1}} G^{(i+1)} = \gen[G]{\Delta_{i-1} \cup \Xi} G^{(i+1)}$ by \cref{lem:derived-series}.

  As the next step, we choose $\Theta \sse \{ [g,h] : g,h \in \Delta_{i-1} \cup \Xi \}$ minimally with the property that $\gen[G]{\Theta} G^{(i+1)} = [\Delta_{i-1} \cup \Xi, \Delta_{i-1} \cup \Xi] G^{(i+1)}$ and, hence, $G^{(i)} = \ncl[G]{\Theta} G^{(i+1)}$ by \cref{lem:derived-series}.
  By minimality of $\Theta$, we have $\abs{\Theta} \leq \log {\lvert G^{(i)} / G^{(i+1)}\rvert}$.
  Moreover, since $\tilde c^\infty_G(t; \Delta_{i-1} \cup \Xi) \leq 7$ for every $t \in \Theta$, this yields $\tilde c^\infty_G(\Theta; \Delta_{i-1} \cup \Xi) \leq 7 \cdot \log {\lvert G^{(i)} / G^{(i+1)}\rvert}$.

  Finally, we use \cref{lem:normal-closure} to obtain $\Delta_{i} \sse G^{(i)}$ with $\gen[G]{\Delta_i} G^{(i+1)} = \ncl[G]{\Theta} G^{(i+1)} = G^{(i)}$ and $\tilde c^\infty_G(\Delta_{i}; \Theta \cup \Sigma) \leq 5 \cdot \log {\lvert G^{(i)}/G^{(i+1)} \rvert}$.
  Combining the straight-line costs yields
  \[
    \begin{multlined}[b]
    \tilde c^\infty_G(\Delta_i; \Delta_{i-1} \cup \Sigma) 
    \leq \tilde c^\infty_G(\Delta_i; \Theta \cup \Sigma) + \tilde c^\infty_G(\Theta; \Delta_{i-1} \cup \Xi) + \tilde c^\infty_G(\Xi; \Delta_{i-1} \cup \Sigma)\\
    \leq 12 \cdot \log {\lvert G^{(i)} / G^{(i+1)} \rvert} + 5 \cdot \log {\lvert G^{(i-1)} / G^{(i+1)} \rvert} 
    \leq 17 \cdot \log {\lvert G^{(i-1)} / G^{(i+1)} \rvert}. 
  \end{multlined}\qedhere
  \]
\end{proof}

Unfortunately, the construction in \cref{pro:derived-series} does not seem to adapt well to the case of bounded width straight-line programs.
The problem is that the commutators $[g,h]$ formed in the construction have arguments $g$ and $h$ that are essentially independent of one another.
Since the arguments $g$ and $h$ may themselves be commutators, nested to potentially unbounded depth, this rules out computation with a bounded number of registers. 

This issue may be circumvented with an alternative construction:
given subsets $\Delta, \Sigma \sse G$ with $\Sigma$ generating $G$, we let $k = \big\lceil \log \abs{G} \big\rceil$ and define sets
\[
  \commop[i+1]{\Delta}{\Sigma} = \normalop{\{ [g,\tilde g] : g \in \commop[i]{\Delta}{\Sigma}, \tilde g \in \normalop{g}{\Sigma}{k} \}}{\Sigma}{k}
\]
for $i \geq 0$ with $\commop[0]{\Delta}{\Sigma} = \normalop{\Delta}{\Sigma}{k}$.
Recall from \cref{lem:normal-closure} that $\commop[0]{\Delta}{\Sigma}$ generates the normal subgroup $\ncl[G]{\Delta}$ and, hence, so does the union $\commop[\ast]{\Delta}{\Sigma} = \bigcup_{i \geq 0} \commop[i]{\Delta}{\Sigma}$.

Note also that $\commop[i]{\Delta}{\Sigma} \sse G^{(i)}$.
Therefore, if $G$ is solvable, then $\commop[i]{\Delta}{\Sigma} = 1$ for all $i \geq k$, so that we then have $\commop[\ast]{\Delta}{\Sigma} = \commop[0]{\Delta}{\Sigma} \cup \dots \cup \commop[k]{\Delta}{\Sigma}$.

\begin{lemma}\label{lem:pcc-efficient}
  If $G$ is solvable, then $c^3_G(t; \Delta \cup \Sigma) \in \mathcal{O}(\log^2 \abs{G})$ for all $t \in \commop[\ast]{\Delta}{\Sigma}$.
\end{lemma}
\begin{proof}
  Let $t = [g, \tilde g]^u \in \commop[i]{\Delta}{\Sigma}$ where $g \in \commop[i-1]{\Delta}{\Sigma}$ and $\tilde g = g^v$ with $u,v \in \Sigma^{\leq k}$.

  To prove the required bound on the straight-line cost, it suffices to show that $t$ can be computed by a straight-line program over $\Delta \cup \Sigma$ of width three and length $\mathcal{O}(\log \abs{G})$ under the additional premiss that $g$ is already contained in one of the registers.
  To do so, we always keep $g$ in this register $r_g$.
  We use one further register $r_a$ as an accumulator, which should hold the value $t$ at the end of execution.
  The third register $r_u$ is a utility register used for loading generators and for fast exponentiation (in particular, the computation of inverses).
  We allow for $r_a$ and $r_u$ to exchange these roles during the computation as necessary.

  It remains to observe that $t = u^{-1}g^{-1}v^{-1}g^{-1}vgv^{-1}gvu$ can be computed from $g$ and $\Sigma$ by successively inverting the accumulator or multiplying the accumulator from the left or right by $g$ or elements from $\Sigma$, and this requires only a constant number of inversions.
\end{proof}

\begin{lemma}\label{lem:pcc-effective}
  If $G$ is solvable, then $\commop[\ast]{\Delta}{\Sigma} $ is a polycyclic generating set of $\ncl[G]{\Delta}$.
\end{lemma}
\begin{proof}
	We proceed by induction on the size of the group $\ncl[G]{\Delta}$.
  We have $\ncl[G]{\Delta} = \prod_{g \in \Delta} \ncl[G]{g}$, and if $\ncl[G]{\Delta} \neq \ncl[G]{g}$ for some $g \in \Delta$, then $\commop[\ast]{\{g\}}{\Sigma}$ is a polycyclic generating set of the normal subgroup $\ncl[G]{g}$ by induction.
  Hence, if $\ncl[G]{\Delta} \neq \ncl[G]{g}$ holds for all $g \in \Delta$, then $\commop[\ast]{\Delta}{\Sigma} = \bigcup_{g \in \Delta} \commop[\ast]{\{g\}}{\Sigma}$ is a polycyclic generating set of $\ncl[G]{\Delta}$.
  Otherwise, we may as well assume that $\Delta = \{ g \}$.
  By \cref{lem:normal-closure}, $N \coloneqq \ncl[G]{g}$ is generated by $\normalop{g}{\Sigma}{k}$.
  Hence, $N' = \ncl[G]{[\normalop{g}{\Sigma}{k}, \normalop{g}{\Sigma}{k}]}$ by \cref{lem:derived-series} and, therefore, the set
  \[
    \Delta' \coloneqq \big\{ [h, \tilde h] : h \in \commop[0]{\{g\}}{\Sigma}, \tilde h \in \normalop{g}{\Sigma}{k} \big\} \supseteq \big\{ [h,h'] : h,h' \in \normalop{g}{\Sigma}{k} \big\}
  \]
  satisfies $N' = \ncl[G]{\Delta'}$.
  By induction, $\commop[\ast]{\Delta'}{\Sigma}$ is a polycyclic generating set of $N'$.
  Now, observe that $\commop[0]{\Delta'}{\Sigma} = \normalop{\Delta'}{\Sigma}{k} = \commop[1]{\Delta}{\Sigma}$.
  Therefore, the set $\commop[\ast]{\Delta}{\Sigma} = \commop[0]{\Delta}{\Sigma} \cup \commop[\ast]{\Delta'}{\Sigma}$ is indeed a polycyclic generating set of $N = \ncl[G]{\Delta}$.
\end{proof}

\begin{proof}[Proof of \cref{thm:solvable-groups}]
  Let $G \in \pv{H} \sse \Gsol$, $\Sigma \sse G$ a generating set, and $t \in G$.
  For the case of straight-line programs of unbounded width, \cref{pro:derived-series} yields a polycyclic generating set $\Delta \sse G$ with $\tilde c^\infty_G(\Delta; \Sigma) \in \mathcal{O}(\log \abs{G})$.
  By \cref{lem:subnormal_SLP}, applied with Abelian subquotients, we obtain $c^\infty_G(t; \Delta) \in \mathcal{O}(\log \abs{G})$.
  Hence, $\tilde c^\infty_G(t; \Sigma) \leq c^\infty_G(t; \Delta) + \tilde c^\infty_G(\Delta; \Sigma) \in \mathcal{O}(\log \abs{G})$.
  This results in the estimate $\tilde C^\infty_{\pv{H}}(N) \in \mathcal{O}(\log N)$.
  In turn, \cref{lem:inverses_SLP} yields $C^\infty_{\pv{H}}(N) \in \mathcal{O}(\log N)$.
	
  For bounded width, $\Delta = \commop[\ast]{\Sigma}{\Sigma}$ is a polycyclic generating set of $G$ by \cref{lem:pcc-effective}.
  By \cref{lem:subnormal_SLP}, applied with \emph{cyclic} subquotients, we obtain $c^3_G(t; \Delta) \in \mathcal{O}(\log \abs{G})$.
  Hence,
  \[
    c^5_G(t; \Sigma) \leq c^3_G(t; \Delta) \cdot \max_{t' \in \Delta} c^3_G(t'; \Delta) \in \mathcal{O}(\log \abs{G}) \cdot \mathcal{O}(\log^2 \abs{G}) \sse \mathcal{O}(\log^3 \abs{G}) 
  \]
  where the first inequality is due to \cref{lem:augmented-generating-set} and the second uses \cref{lem:pcc-efficient}.
  In fact, we can obtain $c^4_G(t; \Sigma) \in \mathcal{O}(\log^3 \abs{G})$, since we can reuse the utility register from the computation of the elements $t' \in \Delta$ as the additional register for  the fast exponentiation performed in the computation for each subquotient in \cref{lem:subnormal_SLP}.
    This shows $C^4_{\pv{H}}(N) \in \mathcal{O}(\log^3 N)$.
\end{proof}

\section{Completely Regular Semigroups}\label{sec:completely_regular}

In this section, we classify the pseudovarieties of completely regular semigroups that afford efficient compression.
Recall that a semigroup $S$ is \emph{completely regular} if $S$ is a union of groups, and that finite such semigroups form a pseudovariety, which we denote by $\pv{CR} = \pvI{x \approx x^{\omega+1}}$.

\medskip

A semigroup $S$ is a \emph{band of groups} if, for some band $B$, there is a partition $S = \bigcup_{\alpha \in B} S_\alpha$ into disjoint subgroups $S_\alpha \leq S$ such that $S_{\alpha} S_{\beta} \sse S_{\alpha \beta}$ for all $\alpha,\beta \in B$.
Bands of groups form an important subclass of the completely regular semigroups.
Imposing restrictions on the band $B$ involved in the above decomposition naturally leads to the classes of \emph{semilattices of groups} -- commonly referred to as \emph{Clifford semigroups} -- and \emph{normal bands of groups}.

The pseudovariety of normal bands $\pv{NB} = \pvI{x^2 \approx x, uxyv \approx uyxv}$ plays an important role here, as it is the largest permutative pseudovariety of bands~\cite[Theorem~10]{YamadaKimura1958} and the largest pseudovariety of bands containing neither of the obstructions $\pv{LRB}$ nor $\pv{RRB}$~\cite{Biryukov1970,Fennemore1971,Gerhard1970}.

The finite normal bands of groups form a pseudovariety, which we denote by $\pv{G} \pvM \pv{NB}$.\footnote{We will only make use of this description as a Mal'cev product for notational purposes.}
If the involved groups are moreover confined to some pseudovariety $\pv{H} \sse \pv{G}$, then we obtain the pseudovariety $\pv{H} \pvM \pv{NB}$ consisting of all finite normal bands of $\pv{H}$-groups.
Crucially, results on efficient compression for groups in $\pv{H}$ transfer to semigroups in $\pv{H} \pvM \pv{NB}$.

\begin{proposition}\label{pro:normal-bands-of-groups}
  Let $\pv{V} \sse \pv{H} \pvM \pv{NB}$ for some pseudovariety $\pv{H} \sse \pv{G}$. 
  For all $2 \leq w \leq \infty$, 
  \vspace{-1ex}
  \[
    C^{w+2}_{\pv{V}}(N) \in \mathcal{O}(C^w_\pv{H}(N) + \log N)
    \quad\text{and}\quad
    C^{w+1}_{\pv{V}}(N) \in \mathcal{O}(C^w_\pv{H}(N) \cdot \log N).
  \]
\end{proposition}

Note that the first of these bounds is, except for the increase in width, essentially optimal.
Indeed, if $\pv{H} \sse \pv{G}$ is not the trivial pseudovariety, then $C^\infty_{\pv{H}}(N) \in \Omega(\log N)$ by \cref{lem:locally-trivial}.

For normal bands of solvable groups, we obtain the following due to \cref{thm:solvable-groups}.

\begin{corollary}\label{cor:normal-bands-of-solvable-groups}
  Let $\pv{V} \sse \Gsol \pvM \pv{NB}$ be a pseudovariety.
  Then 
  \vspace{-1ex}
  \[
    C^\infty_{\pv{V}}(N) \in \mathcal{O}(\log N)
    \quad\text{and}\quad
    C^6_{\pv{V}}(N) \in \mathcal{O}(\log^3 N),\;
    C^5_{\pv{V}}(N) \in \mathcal{O}(\log^4 N).
  \]
\end{corollary}

Our proof of \cref{pro:normal-bands-of-groups} relies on the following construction of generating sets for the subgroups involved in a decomposition as a normal band of groups.

\begin{lemma}\label{lem:normal-bands-of-groups_generators}
  Let $S = \bigcup_{\alpha \in B} S_\alpha$ be a finite normal band of groups generated by a set $\Sigma \sse S$.
   Then $S_\alpha$ is generated by
 $\Sigma_\alpha = \{ e_\alpha s e_\alpha \in S : s \in \Sigma^{\leq 2}, s \geq_{\mathcal{J}} e_\alpha \}$ where $e_\alpha = e^2_\alpha \in S_\alpha$.
\end{lemma}

In the above, $s \geq_{\mathcal{J}} e_\alpha$ refers to Green's preorder $\geq_{\mathcal{J}}$, meaning that $e_\alpha = usv$ holds for some $u,v \in S \cup \{ 1 \}$.
Note that $s \in S_\beta$ satisfies $s \geq_{\mathcal J} e_\alpha$ if and only if $\beta \geq_{\mathcal J} \alpha$ in $B$.
Since $B$ is a normal band, this is furthermore equivalent to $\alpha \beta \alpha = \alpha$.
Hence, $\Sigma_\alpha \sse S_\alpha$.
\begin{proof}
  Suppose that $t \in S_\alpha$ and $t = s_1 \cdots s_n$ with $s_1, \dots, s_n \in \Sigma$.
  Then $e_\alpha \mathrel{\mathcal{H}} t \leq_{\mathcal{J}} s_i$ and, thus, $e_\alpha \leq_{\mathcal{J}} s_i$ for all $1 \leq i \leq n$.
  Considering the expression $t = e_\alpha s_1 \cdots s_n e_\alpha$, we note that every prefix of the form $e_\alpha s_1 \cdots s_i$ belongs to $S_{\alpha \beta_i}$ where $s_i \in S_{\beta_i}$ and, similarly, that every suffix of the form $s_i \cdots s_n e_\alpha$ belongs to $S_{\beta_i \alpha}$.
  Upon inserting the elements $e_{\alpha\beta_i}$~and~$e_{\beta_i \alpha}$, we obtain the expression $t = e_\alpha s_1 e_{\alpha\beta_1} \cdots e_{\beta_i\alpha} s_i e_{\alpha\beta_i} \cdots e_{\beta_n \alpha} s_n e_\alpha,$
  wherein the pre- and suffixes of the form $e_\alpha \cdots e_{\beta_i\alpha}$ and $e_{\alpha\beta_i} \cdots e_\alpha$ belong to $S_\alpha$.
  Inserting the element $e_\alpha$ yields
  \[
    t =  (e_\alpha s_1 e_{\alpha }) e_{\alpha\beta_1} \cdots e_{\beta_i\alpha} (e_\alpha s_i e_\alpha) e_{\alpha\beta_i} \cdots e_{\beta_n \alpha} (e_\alpha s_n e_\alpha).
  \]
  
  We have thus written $t$ as a product of elements $e_\alpha s_i e_\alpha \in \Sigma_\alpha$ as well as elements of the form $e_{\alpha \beta_i}e_{\beta_j \alpha}$.
  The latter also belong to $\gen[S]{\Sigma_\alpha} \subseteq S_\alpha$, as such an element can be written as
  \begin{align*}
    e_{\alpha\beta_i}e_{\beta_j\alpha} = (e_{\alpha}s_i)^\omega (s_j e_{\alpha})^\omega 
    &= (e_{\alpha}s_ie_{\alpha})^{\omega-1} (e_\alpha s_is_j e_{\alpha}) (e_{\alpha}s_j e_{\alpha})^{\omega-1}. \qedhere
\end{align*}
\end{proof}

\begin{proof}[Proof of \cref{pro:normal-bands-of-groups}]
  Consider a semigroup $S \in \pv{V}$, a generating set $\Sigma \sse S$, and some element $t \in S$.
  Since $\pv{V} \sse \pv{H} \pvM \pv{NB}$, the semigroup $S$ admits a decomposition $S = \bigcup_{\alpha \in B} S_\alpha$ into subgroups $S_\alpha \leq S$ belonging to $\pv{H}$ for some normal band $B$ such that the map $\pi \colon S \to B$ sending every element of a subgroup $S_\alpha$ to the corresponding $\alpha \in B$ is a homomorphism.

  \smallskip

  Let $e_\alpha$ be the neutral element the subgroups $S_\alpha \leq S$ for some $\alpha \in B$.
  We claim that this element can be efficiently computed by a straight-line program, viz., $c^2_S(e_\alpha; \Sigma) \in \mathcal{O}(\log \abs{S})$.
  To see this, first note that the pseudovariety $\pv{NB} = \pvI{x^2 \approx x, uxyv \approx uyxv}$ is permutative and, therefore, $c^2_B(\alpha; \pi(\Sigma)) \in \mathcal{O}(\log \abs{B}) \sse \mathcal{O}(\log \abs{S})$ by \cref{pro:permutative}.
  Next, take any straight-line program in $B$ over $\pi(\Sigma)$ that achieves this bound and lift it to a straight-line program in $S$ over $\Sigma$.
  The lifted straight-line program computes some element $t_\alpha \in S_\alpha$, and we then obtain $e_\alpha = t_\alpha^\omega$ using fast exponentiation in $\mathcal{O}(\log \abs{S_\alpha}) \sse \mathcal{O}(\log \abs{S})$ instructions.

  \smallskip

  Let us now fix the unique $\alpha \in B$ with $t \in S_\alpha$, and denote by $\Sigma_\alpha$ the generating set of $S_\alpha$ from \cref{lem:normal-bands-of-groups_generators}.
  Further, let us fix a straight-line program $\mathcal{A}$ computing $t$ over $\Sigma_\alpha \sse S_\alpha$ of width at most $w$ and length $c^w_{S_\alpha}(t; \Sigma_\alpha)$.
  For the case that two additional registers are permitted, we modify this straight-line program as follows.
  First, we compute $e_\alpha \in S_\alpha$ as described above, and keep it in one of the additional registers throughout the computation.
  From this point on, we emulate $\mathcal{A}$ on the original registers and, whenever we encounter an assignment $r_i \gets s$ for some generator $s \in \Sigma_\alpha$, we use the register $r_i$ and the additional register not holding $e_\alpha$ to compute $s = e_\alpha s_1 e_\alpha$ or $s = e_\alpha s_1s_2 e_\alpha$ with $s_1, s_2 \in \Sigma$ as appropriate.
  This yields the estimate $c^{w+2}_S(t; \Sigma) \leq 5 \cdot c^{w}_{S_\alpha}(t; \Sigma_\alpha) + \mathcal{O}(\log \abs{S})$.

  We proceed similarly if only one additional register is permitted, but temporarily overwrite the element $e_\alpha$ when producing a generator of the form $e_\alpha s_1s_2 e_\alpha$ with $s_1, s_2 \in \Sigma$.
  Specifically, suppose that $r_e$ holds $e_\alpha$ and that we want to emulate the assignment $r_i \gets e_\alpha s_1 s_2 e_\alpha$.
  Then we first perform the instructions $r_i \gets s_1$; $r_i \gets r_e \cdot r_i$; $r_e \gets s_2$; $r_i \gets r_i \cdot r_e$.
  At this point, the target register $r_i$ holds $e_\alpha s_1 s_2$.
  Using one of the original registers $r_{j} \neq r_i$, which we may assume to contain an element $s \in S_\alpha$, we recompute $e_\alpha = s^\omega$ in the register $r_e$ using fast exponentiation (i.e., $r_e \gets r_j \cdot r_j$ followed by a suitable sequence of the instructions $r_e \gets r_e \cdot r_e$ and $r_e \gets r_e \cdot r_j$.)
  Finally, $r_i \gets r_i \cdot r_e$ produces $e_\alpha s_1s_2 e_\alpha$ in the target register.
  This constructions yields the estimate $c^{w+1}_S(t; \Sigma) \in \mathcal{O}(c^{w}_{S_\alpha}(t; \Sigma_\alpha) \cdot \log \abs{S})$.
\end{proof}

Summarizing the preceding results, we arrive at the following characterization for pseudovarieties of completely regular semigroups with efficient compression.

\begin{theorem}\label{thm:main-cregular}
  Let $\pv{V} \sse \pv{CR}$ be a pseudovariety.
  The following are equivalent.
  \begin{bracketenumerate}
    \item The pseudovariety $\pv{V}$ affords efficient compression.
    \item The pseudovariety $\pv{V}$ contains neither $\pv{LRB}$ nor $\pv{RRB}$.
    \item The pseudovariety $\pv{V}$ comprises only normal bands of groups; that is, $\pv{V} \sse \pv{G}\pvM\pv{NB}$.
    \item The pseudovariety $\pv{V}$ admits straight-line programs of length $\mathcal{O}(\smash{\log}^2 N)$.
  \end{bracketenumerate}
\end{theorem}

\begin{proof}
  \Cref{lem:obstructions} shows that $(1) \Rightarrow (2)$. 
  The implication $(2) \Rightarrow (3)$ is a well-known result by Rasin~\cite[Proposition~4]{Rasin81}: If $\pv{V} \subseteq \pv{CR}$ and $\pv{V} \cap \pv{B} \subseteq \pv{NB}$, then $\pv{V} \subseteq \pv{G} \pvM \pv{NB}$.
  (Recall that for any pseudovariety $\pv{W} \subseteq \pv{B}$ either $\pv{W} \subseteq \pv{NB}$ holds, or $\pv{W}$ contains $\pv{LRB}$ or $\pv{RRB}$.)

  \Cref{pro:normal-bands-of-groups} and the Reachability Lemma of Babai and Szemer\'edi~\cite[Theorem~3.1]{BabaiS84} combine to show the implication $(3) \Rightarrow (4)$.
  Finally, $(4) \Rightarrow (1)$ holds by definition.
\end{proof}

\section{General Semigroups}\label{sec:general}

In this section, we finally prove our main theorem (\cref{thm:main-intro}).
The following restatement also reveals the close connection between the semigroup and group cases.

\begin{theorem}\label{thm:main}
  Let $\pv{V}$ be a pseudovariety, and let $\pv{H} = \pv{V} \cap \pv{G}$.
  The following are equivalent.
  \begin{bracketenumerate}
    \item The pseudovariety $\pv{V}$ affords efficient compression.
    \item The pseudovariety $\pv{V}$ contains neither $\pv{LRB}$, $\pv{RRB}$, nor $\pv{T}$.
    \item The pseudovariety $\pv{V}$ satisfies the following for every $2 \leq w \leq \infty$:
      \[
        C^{w+3}_{\pv{V}}(N) \in \mathcal{O}(C^w_{\pv{H}}(N) + \log N)
        \quad\text{and}\quad
        C^{w+2}_{\pv{V}}(N) \in \mathcal{O}(C^w_{\pv{H}}(N) \cdot \log N).
      \]
      In particular, $\pv{V}$ admits straight-line programs of length $\mathcal{O}(\smash{\log}^2 N)$. 
  \end{bracketenumerate}
\end{theorem}

Our proof of \cref{thm:main} is based on an observation on the structure of a pseudovariety~$\pv{V}$ with $\pv{T} \not\sse \pv{V}$ that was recently obtained by the second author~\cite[Theorem~A, Theorem~B]{Thumm2025}.

\begin{theorem}\label{thm:excluding-T}
  Let $\pv{V}$ be a pseudovariety.
  If $\pv{T} \not\sse \pv{V}$, then \mbox{one of the following holds.}
  \begin{bracketenumerate}
    \item There exist $1 \leq i \leq j \leq n$ such that
      $\pv{V} \models x_1 \cdots x_n \approx x_1 \cdots x_{i-1} (x_i \cdots x_j)^{\omega + 1} x_{j+1} \cdots x_n$.
    \item There exists $\sigma \in \mathrm{Sym}(n)$ with $\sigma \neq \mathrm{id}$ such that $\pv{V} \models x_1 \cdots x_n \approx x_{\sigma(1)} \cdots x_{\sigma(n)}$.
  \end{bracketenumerate}
\end{theorem}

Notably, for monoids all identities in the first item of \cref{thm:excluding-T} are equivalent to the defining identity $x \approx x^{\omega + 1}$ of completely regular monoids, and all identities in the second item are equivalent to the defining identity $xy \approx yx$ of commutative monoids.
Additionally excluding $\pv{LRB}$ and $\pv{RRB}$ narrows the first item to the pseudovarieties of Clifford monoids by Rasin's result~\cite[Proposition~4]{Rasin81}.
In this way, we obtain an alternative proof of Fleischer's classification of pseudovarieties of monoids affording efficient compression \cite[Theorem~4.13]{Fleischer19diss}.

For semigroups, we will also need the following observation, which is a direct consequence of Volkov's characterization \cite[Theorem~1.1]{Volkov00} (see also \cite[Lemma~18]{Thumm2025}).

\begin{lemma}\label{lem:cregular_subsemigroup}
  Let $\pv{V}$ be a pseudovariety with $\pv{T} \not\sse \pv{V}$. 
  For every semigroup $S \in \pv{V}$, the set of its completely regular elements $I(S) = \{ s \in S : s = s^{\omega + 1}\}$ is a subsemigroup of $S$.
\end{lemma}

\begin{proof}[Proof of \cref{thm:main}]
  \Cref{lem:obstructions} shows that $(1) \Rightarrow (2)$.
  Since the addendum in item $(3)$ follows from Babai and Szemer\'edi's Reachability Lemma, we have that $(3) \Rightarrow (1)$ by definition.

  \smallskip

  Let us thus assume that $(2)$ holds. 
  Recall from \cref{thm:excluding-T} that $\pv{V}$ satisfies the identity 
  \[
    x_1 \cdots x_n \approx x_1 \cdots x_{i-1} (x_i \cdots x_j)^{\omega + 1} x_{j+1} \cdots x_n
  \]
  for some parameters $1 \leq i \leq j \leq n$, or is permutative otherwise.
  We may assume the former, since otherwise $C^2_{\pv{V}}(N) \in \mathcal{O}(\log N)$ holds by \cref{pro:permutative} and this clearly implies~$(3)$.

  Under this assumption, we have $\pv{V} \sse \pvI{xyz \approx xy^{\omega+1}z} \pvM \pv{N}_k$ for a sufficiently large $k \geq 1$.
  Let $\pv{V}' = \pv{V} \cap \pvI{xyz \approx xy^{\omega+1}z}$, and note that $\pv{V}' \sse \pv{V}$ and $\pv{V} \sse \pv{V}' \pvM \pv{N}_k$.
  Thus, $\pv{V}'$ also satisfies $(2)$ and, in light of \cref{lem:nilpotent}, it suffices to prove that, for all $2 \leq w \leq \infty$, 
  \[
    C^{w+2}_{\pv{V}'}(N) \in \mathcal{O}(C^w_{\pv{H}}(N) + \log N)
    \quad\text{and}\quad
    C^{w+1}_{\pv{V}'}(N) \in \mathcal{O}(C^w_{\pv{H}}(N) \cdot \log N).
  \]

  Consider a semigroup $S \in \pv{V}'$, a generating set $\Sigma \subseteq S$, and some element $t \in S$.
  Without loss of generality, the element $t$ is not a product of generators with two or fewer factors.
  Under these assumptions, the element $t$ can be written as $t = u\, s_1 \cdots s_n\, v$ with $u, s_1, \dots, s_n, v \in \Sigma$.

  For $1 \leq i \leq n$, let $\tilde s_i = s_i^{\omega + 1}$ and note that $\tilde s_i \in I(S)$.
  Further, let $\tilde S \leq S$ be the subsemigroup generated by the set $\tilde\Sigma = \{ \tilde s_1, \dots, \tilde s_n \}$.
  Since $\tilde \Sigma \sse I(S)$ and the latter is a subsemigroup of $S$ by \cref{lem:cregular_subsemigroup}, we in fact have $\tilde S \leq I(S) \leq S$ and, therefore, $\tilde S \in \pv{V}' \cap \pv{CR}$.

  By \cref{thm:main-cregular}, $\tilde S \in \pv{G} \pvM \pv{NB}$.
  Moreover, since $\pv{V'} \cap \pv{G} \sse \pv{H}$, we have that $\tilde S \in \pv{H} \pvM \pv{NB}$.
  By \cref{pro:normal-bands-of-groups}, for every $2 \leq w \leq \infty$, there is a straight-line program $\smash{\tilde{\mathcal{A}}}$ computing $\tilde s \in \tilde S$ over $\tilde \Sigma \sse \tilde S$ with $w(\smash{\tilde{\mathcal{A}}}) = w + 2$ (or $w + 1$) and $\ell(\smash{\tilde{\mathcal{A}}}) \in \mathcal{O}(C^w_{\pv{H}}(N) + \log N)$ (resp.\ $\mathcal{O}(C^w_{\pv{H}}(N) \cdot \log N)$) where $N = \lvert\tilde S\rvert \leq \abs{S}$.
  We modify the straight-line program $\smash{\tilde{\mathcal{A}}}$ by replacing every instruction $r_k \gets \tilde s_i$ with $\tilde s_i \in \tilde \Sigma$ by the instruction $r_k \gets s_i$ where $s_i \in \Sigma$ with $s_i^{\omega + 1} = \tilde s_i$. 
  (In case $s_i \neq s_{i'}$ satisfy $s_i^{\omega + 1} = s_{i'}^{\omega + 1}$, we choose arbitrarily between $s_i$ and $s_{i'}$.)
  The result of this modification is a straight-line program $\mathcal{A}$ over $\Sigma\sse S$ of the same length and width as $\smash{\tilde{\mathcal{A}}}$.
  Let $s \in S$ be the element computed by $\mathcal{A}$ corresponding to $\tilde s \in \tilde S$.
  We claim that $t = u \, s \, v$.

  To see this claim to be true, note that there is a word $\rho(x_1, \dots, x_n)$ over $\{ x_1, \dots, x_n \}$ that evaluates to $\rho(s_1, \dots, s_n) = s$ and $\rho(\tilde s_1, \dots, \tilde s_n) = \tilde s$.
  Due to the identity $xyz \approx xy^{\omega+1} z$, 
  \[
    t = u \, s_1 \cdots s_n \, v = u \, \tilde s_1 \cdots \tilde s_n \, v = u \, \tilde s \, v = u \, \rho(\tilde s_1, \dots, \tilde s_n) \, v = u \, \rho(s_1, \dots, s_n) \, v = u \, s \, v
  \] 
  as required.
  Since, moreover, $t = u\, s \, v$ can be computed from $s$ using no additional registers, since $w(\mathcal{A}) \geq 2$, and at most four additional instructions, this completes the proof.
\end{proof}

\begin{remark}
  By reusing one of the additional registers between the constructions in \cref{lem:nilpotent} and \cref{pro:normal-bands-of-groups}, the above bounds can be slightly improved to 
    \vspace{-1ex}
    \[
      C^{w+2}_{\pv{V}}(N) \in \mathcal{O}(C^w_{\pv{H}}(N) + \log N)
      \quad\text{and}\quad
      C^{w+1}_{\pv{V}}(N) \in \mathcal{O}(C^w_{\pv{H}}(N) \cdot \log N).
    \]
\end{remark}

Incorporating this improvement, we obtain the following due to \cref{cor:normal-bands-of-solvable-groups}.

\begin{corollary}\label{cor:solvable_efficient_compression}
  Let $\pv{V}$ be a pseudovariety, $\pv{LRB}, \pv{RRB}, \pv{T} \not\sse \pv{V}$, and $\pv{V} \cap \pv{G} \sse \Gsol$.
  Then 
  \vspace{-1ex}
  \[
    C^\infty_{\pv{V}}(N) \in \mathcal{O}(\log N)
    \quad\text{and}\quad
    C^6_{\pv{V}}(N) \in \mathcal{O}(\log^3 N),\;
    C^5_{\pv{V}}(N) \in \mathcal{O}(\log^4 N).
  \]
\end{corollary}

\section{The Membership Problem}\label{sec:membership}

The \emph{membership problem} is one of the most fundamental decision problems for algebraic structures.
In the case of semigroups, this problem is given as follows.
\smallskip
\begin{decproblem}
  \iitem A semigroup $S$, a subset $\Sigma \sse S$, and an element $t \in S$.
  \qitem Is $t$ a member of the subsemigroup $\gen[S]{\Sigma} \leq S$?
\end{decproblem}

Here, we consider a restricted variant \dMemb{\pv{V}} with the additional promise that the subsemigroup\footnote{
Be aware that in some related work, including \cite{Fleischer19diss,Fleischer22}, not only the subsemigroup $\gen{\Sigma}$ is restricted to the pseudovariety $\pv{V}$, but also the surrounding semigroup $S$.
As we require a weaker promise, we obtain a stronger algorithmic result (\cref{cor:main_membership_restated}) than with the formulation from \cite{Fleischer19diss,Fleischer22}.}
$\gen[S]{\Sigma} \leq S$ belongs to some fixed pseudovariety $\pv{V}$.
This includes the general problem as \dMemb{\pv{S}} and allows for a more fine-grained analysis of the problem's complexity.

There are different ways to represent the semigroup $S$ and its elements as part of the input. 
Our main focus is on the \emph{Cayley table model} where $S$ is given as multiplication table and its elements as indices into its rows and columns; in this case, we denote the membership problem by \dMemb[CT]{\pv{V}}.
Other forms of input are considered at the end of this section.

Note that, in general, the membership problem \dMemb[CT]{\pv{S}} is \NL-complete \cite{JonesLL76}; in fact, this already holds for \dMemb[CT]{\pv{T}} \cite[Theorem 5.6]{Fleischer19diss}.
On the other hand, our results~--~see \cref{cor:main_membership_restated} below~--~imply that, for any pseudovariety $\pv{V}$ with $\pv{LRB},\pv{RRB}, \pv{T} \not\sse \pv{V}$, the corresponding membership problem \dMemb[CT]{\pv{V}} cannot even be hard under \ACz-reductions for any complexity class containing \textsc{parity}~-- such as \NL.

The following observation shows that efficient compression via straight-line programs gives rise to nontrivial complexity upper bounds for the membership problem.
Both items can be found in the dissertation of Fleischer~\cite[Corollary~5.2, Corollary~5.3]{Fleischer19diss}.
However, using the techniques by Collins, Grochow, Levet, and the third author~\cite{CollinsGLW25} (where this fact has also been used implicitly), we can give a simpler proof of the second item.

\begin{lemma}[Fleischer]\label{lem:complexity}
  Let $\pv{V}$ be a pseudovariety, and let $2 \leq w < \infty$.
  \begin{bracketenumerate}
    \item If $C^\infty_{\pv{V}}(N) \in \mathcal{O}(\polylog N)$, then \dMemb[CT]{\pv{V}} is in $\NPOLYLOGTIME \sse \qACz$.
    \item If $C^w_{\pv{V}}(N) \in \mathcal{O}(\polylog N)$, then \dMemb[CT]{\pv{V}} is in $\NTISP(\polylog n, \log n) \sse \FOLL$.
  \end{bracketenumerate}
\end{lemma}
\begin{proof}
  To certify membership of an element $t \in S$ in a subsemigroup $\gen[S]{\Sigma} \leq S$, we can use a suitable straight-line program $\mathcal{A}$ over $\Sigma$ computing $t$.
  We guess and execute $\mathcal{A}$, instruction by instruction, on the random access machine and accept once some register is assigned $t$. 
  
  For the bounded-width case, observe that we only need to store a bounded number of elements, which only requires logarithmic space. 
  In this case, the membership problem can thus be solved in $\NTISP(\polylog n, \log n)$, which is a subset of $\FOLL$ \cite[Lemma 2.6]{CollinsGLW25}.
\end{proof}

Using this lemma, we obtain the following from \cref{thm:main} (thus, proving \cref{cor:main_membership}).

\begin{corollary}\label{cor:main_membership_restated}
	Let $\pv{V}$ be a pseudovariety. The following are equivalent.
	\begin{bracketenumerate}
		\item\label{cor_efficient_compression} The pseudovariety $\pv{V}$ affords efficient compression.
    \item The pseudovariety $\pv{V}$ contains neither $\pv{LRB}$, $\pv{RRB}$, nor $\pv{T}$.
		\item\label{cor_mem_npolylog} The membership problem \dMemb[CT]{\pv{V}} is in \NPOLYLOGTIME.
	\end{bracketenumerate}
	 Furthermore, if $\pv{V}$ contains neither $\pv{LRB}$, $\pv{RRB}$, $\pv{T}$, nor any nonsolvable group, then the membership problem \dMemb[CT]{\pv{V}} is in $\NTISP(\polylog n, \log n)$ and, hence, in $\FOLL$. 
\end{corollary}

Note that in some cases even smaller bounds are known: if $\pv{V} \sse \pvL \pv{I}_k$ or $\pv{V} \sse\pv{NB}$, then the membership problem is in \ACz \cite[Proposition 4.5 and Corollary 5.2, Theorem 5.12]{Fleischer19diss}.
\begin{proof}
	The first two items are equivalent by \cref{thm:main} and imply the third by \cref{lem:complexity}.

  	On the other hand, if the problem \dMemb[CT]{\pv{V}} is in \NPOLYLOGTIME, then for any fixed semigroup $S \in \pv{V}$ and generating set $\Sigma \sse S$, a target element $t \in S$ can only depend on a polylogarithmic number of generators as in one computation path the algorithm cannot access more elements during its running time~-- meaning that $t \in \gen[S]{\Sigma'}$ for some $\Sigma' \sse \Sigma$ of polylogarithmic size.
  	In turn, this implies that $\pv{LRB}$, $\pv{RRB}$, $\pv{T} \not \sse \pv{V}$ by \cref{lem:obstructions}.
	
  	Finally, if $\pv{V}$ contains neither $\pv{LRB}$, $\pv{RRB}$, $\pv{T}$, nor any nonsolvable group, then the problem \dMemb[CT]{\pv{V}} is in $\NTISP(\polylog n, \log n) \sse \FOLL$ by \cref{cor:solvable_efficient_compression} and \cref{lem:complexity}.
\end{proof}

\Cref{cor:main_membership_restated} confirms, in particular, the conjecture by Barrington, Kadau, Lange, and McKenzie \cite{BarringtonKLM01} that solvable groups have their membership problem in \FOLL.
Moreover, it completely reduces Fleischer's question~\cite{Fleischer22} whether all classes of semigroups that afford efficient compression have their membership problem in \FOLL to the case of groups.

\medskip

For membership in the \emph{transformation model}, the semigroup $S$ is the semigroup $\mathcal{T}_X$ of all maps from some finite set $X$ to itself, with elements given in a point-wise representation.

\begin{corollary}\label{cor:membership_transformation}
	Let $\pv{V}$ be a pseudovariety of semigroups with $\pv{LRB}$, $\pv{RRB}$, $\pv{T} \not \sse \pv{V}$.
  Then the membership problem for $\pv{V}$-semigroups in the transformation model is in \NP.
\end{corollary}

Note that this corollary is far from being optimal.
Indeed, it should be rather easy to show that for $\pv{G} \pvM \pv{NB}$ the membership problem in the transformation model is in \NC (building on \cite{BabaiLS87}).
Moreover, Fleischer and the authors already showed that this holds for Clifford semigroups and, more generally, for strict inverse semigroups \cite[Theorem B]{FleischerSTW25}. 

On the other hand, \cref{cor:membership_transformation} has some interesting consequences to the minimum generating set problem (given a semigroup and a number $k$, decide whether the semigroup can be generated by at most $k$ elements), the problem of solving equations, and the isomorphism problem.
Indeed, if $\pv{LRB}$, $\pv{RRB}$, $\pv{T} \not \sse \pv{V}$, then the former two problems can be solved in \NP for semigroups from $\pv{V}$ by simply guessing a suitable generating set (resp.\ a solution to the equations) and then checking its validity using the algorithm for the membership problem.
Moreover, using a similar approach the isomorphism problem can be solved in $\SigkPTIME{2}$, that is, in the second level of the polynomial-time hierarchy.

\section{Conclusion}

In this work, we classified those pseudovarieties of semigroups that afford efficient compression via straight-line programs and also considered the case of efficient compression via bounded-width straight-line programs.
In the absence of nonsolvable groups, we obtained \FOLL algorithms for the membership problem for such pseudovarieties, thereby solving open problems from \cite{Fleischer22} and \cite{BarringtonKLM01}.
We conclude with the following questions.

\begin{question}
  Suppose that the pseudovariety $\pv{V}$ affords efficient compression.
  \begin{itemize}
    \item Does $\pv{V}$ admit straight-line program of \emph{logarithmic} length and unbounded width?
    \item Does $\pv{V}$ admit straight-line program of polylogarithmic length and \emph{bounded} width?
    \item Does $\pv{V}$ admit straight-line program of \emph{logarithmic} length and \emph{bounded} width?
  \end{itemize}
\end{question}

We note that \cref{thm:main} reduces all three of these questions to pseudovarieties of groups.
The second one is particularly interesting, as an affirmative answer would imply that all such pseudovarieties admit \FOLL algorithms for their membership problem.

\begin{question}
  Which pseudovarieties have their membership problem in \ACz?
\end{question}

We suspect that some of our results and techniques (in particular, \cref{lem:nilpotent} and \cref{thm:bounded-diameter}) might be useful in making progress towards resolving this question.
In contrast, addressing the following question appears to be a significantly more ambitious endeavor.

\begin{question}
  Does the membership problem exhibit a \qACz vs. \NL-complete dichotomy?
\end{question}

Finally, we want to ask to what extent the methods of this work can be transferred to study compression in other algebraic structures such as rings and quasigroups.

\bibliography{references}

\end{document}